\documentclass[10pt]{elsarticle}
\usepackage{amsmath}                                                        % need for subequations
\numberwithin{equation}{section}
\usepackage{graphicx}                                                       % need for figures
\usepackage{subfigure}
\usepackage{enumitem}                                                    % need for subfigures
\usepackage{hyperref}
\usepackage{amssymb}                                                        % gives you \mathbb{} font
\usepackage[mathscr]{eucal}                                             % gives you \mathscr font
\usepackage{cancel}                                                             % gives you the ability to visibly cross out terms in equations}
\usepackage[normalem]{ulem}                                                                 % gives you \sout
\usepackage{pstricks}
\usepackage{rotating}
\usepackage{lscape}
\usepackage[paperwidth=8.5in,paperheight=11in,top=1.00in, bottom=1.00in, left=1.00in, right=1.00in]{geometry}
\usepackage{mathtools}                                                      % need for `show only references'
\mathtoolsset{showonlyrefs=true}                                    % only equations which are labeled AND referenced will be numbered.
\usepackage{fixltx2e,amsmath}                                           % Supposedly, this allows one to use \eqref{} in \caption{}.
\MakeRobust{\eqref}

\usepackage{mathdots}
\usepackage{amsthm}                                                             % need for theorem-proof environment
\allowdisplaybreaks                                                             % allows page breaks for long equations
\theoremstyle{plain}
\newtheorem{theorem}{Theorem}
\numberwithin{theorem}{section}

\newtheorem{lemma}[theorem]{Lemma}                              % [theorem] ==> theorems and lemmas will share a counter
\newtheorem{proposition}[theorem]{Proposition}

\theoremstyle{definition}
\newtheorem{definition}[theorem]{Definition}
\newtheorem{example}[theorem]{Example}
\newtheorem{notation}[theorem]{Notation}
\newtheorem{remark}[theorem]{Remark}
\newtheorem{assumption}[theorem]{Assumption}

\def \barG {\Gamma}
\def \s {{\sigma}}
\def \d {{\delta}}
\def \g {{\gamma}}
\def \a {{\alpha}}
\def \b {{\beta}}
\def \l {{\lambda}}

\def \ggg {{\beta}}
\def \z {{\zeta}}
\def \R  {{\mathbb {R}}}
\def \x {{\xi}}
\def \g {{\gamma}}

\def \eps {{\varepsilon}}
\def \m {{\mu}}
\newcommand\N{\mathbb{N}}

\newcommand\Eb{\mathbb{E}}

\newcommand\Nb{\mathbb{N}}
\newcommand\G{\Gamma}
\newcommand\p{\partial}
\newcommand\Act{{\mathscr{A}}}
\newcommand\at{{a}}
\newcommand\Ac{\mathscr{A}}

\newcommand\Gc{\mathscr{G}}

\newcommand\Lc{\mathscr{L}}
\newcommand\Mc{\mathscr{M}}

\newcommand\Kc{\mathscr{K}}

\newcommand\xb{\bar{x}}

\newcommand\zb{\bar{z}}

\newcommand\Cv{\mathbf{C}}
\newcommand\Mv{\mathbf{M}}
\def \TT {\mathbf{T}}

\def \p {{\partial}}

\newcommand{\norm}[1]{\left\|{#1}\right\|}

\newcommand\dd{d}

\newcommand\Nzero{\mathbb{N}_{0}}
\newcommand\Rdd{\mathbb{R}\times\mathbb{R}^{d}}
\newcommand\formaldeg{m}%{\mathfrak{d}}
\newcommand\Ndzero{\mathbb{N}^d_{0}}
\newcommand\Rd{\mathbb{R}^{d}}
\def \phi {{\varphi}}
\def \eps {{\varepsilon}}
\begin{document}
\begin{frontmatter}
\title{Intrinsic expansions for averaged diffusion processes}
\author[pol]{S. Pagliarani \corref{cor1}\fnref{fn1}}
\ead{stepagliara1@gmail.com}

\author[bol]{A. Pascucci}
\ead{andrea.pascucci@unibo.it}

\author[bol]{M. Pignotti}
\ead{michele.pignotti2@unibo.it}

\address[pol]{DEAMS, Universit\`a di Trieste, Via Tigor 22, 34124 Trieste, Italy}
\address[bol]{Dipartimento di Matematica, Universit\`a di Bologna, Bologna, Italy}

\cortext[cor1]{Corresponding author}

\fntext[fn1]{The author's research was supported by the Chair {\it Financial Risks} of the {\it Risk Foundation} and the {\it  Finance for Energy Market Research Centre}.}

\date{This version: \today}

\begin{abstract}
We show that the convergence rate of asymptotic expansions for solutions of SDEs is higher in the case of {degenerate %(or partial)
diffusion} compared to the elliptic case, i.e. it is higher when the Brownian motion directly acts only along some directions. %on some components of the diffusion.
In the scalar case, this phenomenon was already observed in \cite{Gobet2014_weak} using Malliavin calculus techniques. Here, we provide a general and detailed analysis by {employing} the recent study of intrinsic functional  spaces related to hypoelliptic Kolmogorov operators  in \cite{PPP16}. Applications to finance are discussed, %in particular
in the study of path-dependent derivatives (e.g. Asian options) and in models incorporating dependence on past information.
\end{abstract}
%\maketitle
\end{frontmatter}
%%%%%%%%%%%%%%%%%%%%%%%%%%%%%%%%%%%%%%%
%
%       SECTION: Introduction
%
%%%%%%%%%%%%%%%%%%%%%%%%%%%%%%%%%%%%%%%
\section{Introduction}

%\begin{example}\label{ariasi}
{\noindent We study the asymptotic expansion of the conditional expectation
\begin{equation}\label{eq:cond_expect}
 u(t,x):=\Eb_{t,x} [\varphi(X_T)], %,\qquad  (t,x)\in[0,T]\times D. %0\leq t<T,\ x\in D.
\end{equation}
where $X=(X_{t})_{t\in[0,T]}$ is a continuous $\R^d$-valued Feller process %. %on a filtered space
%%$\left(\Omega,\F,(\F_{s}^{t})_{0\le t\le s\le T},(P_{t,x})\right)$.
%Hereafter we adopt definitions and notations from \cite{FriedmanSDE}.
%
%We assume that $X$ is
and a degenerate diffusion in a sense that will be specified later.

The prototype process we have
in mind is $X=(S,A)$  solution to the SDE
%Let $X=(S,A)$ where
\begin{equation}\label{BS}
  \begin{cases} dS_{t}=\s S_{t}dW_{t}, \\
  dA_{t}=S_{t}dt,
 \end{cases}
\end{equation}
where $W$ is a real Brownian motion. In financial applications, $S$ and $A$ represent the price
and average processes respectively, in the Black\&Scholes model for arithmetic Asian options.
%In this case, we have $d=2$, $p_{0}=1$ and
%the diffusion coefficient is
% $a_{11}(t,x_{1},x_{2})=\s^{2}x_{1}^{2}$, with $x_{1}\in\R_{>0}$, and
%\begin{equation}\label{BKol}
% B=  \begin{pmatrix}
%    0 & 0 \\
%    1 & 0
%  \end{pmatrix}.
%\end{equation}
The infinitesimal generator of $(S,A)$ %{\green (replace $(x_1,x_2)\to (s,a)$?) }
  $$\mathcal{A}_{X}:=
  \frac{\s^{2}s^{2}}{2}\p_{ss}+s\p_{a},\qquad (s,a)\in\R_{>0}\times\R_{>0},$$
%  $$\frac{\s^{2}x_{1}^{2}}{2}\p_{x_{1}x_{1}}+x_{1}\p_{x_{2}},\qquad (x_{1},x_{2})\in\R_{>0}\times\R_{>0},$$
%and clearly coincides, {\it on any domain $D$ compactly contained in $\R_{>0}\times\R_{>0}$,} with
%an operator of the form \eqref{Ac}, that satisfies Assumptions \ref{assA} and \ref{assB}. Clearly,
%more sophisticated multi-dimensional models, including stochastic and local volatility, as well as non-null interest rates, fall
%within the general framework above. %this framework.
is degenerate in two ways: on the one hand, the quadratic form of the second order part is
singular (it has rank one) and, on the other hand, it degenerates completely on the half-line $\{s=0,a>0\}$. %{\it on any domain $D$ compactly contained in $\R_{>0}\times\R_{>0}$,} with
However, for any $0<a<b$, $\mathcal{A}_{X}$ is a hypoelliptic operator on the strip
$D:=\,]a,b[\,\times\R_{>0}$ and %can be extended out of $D$ in order to
coincides on $D$ with an operator that satisfies the H\"ormander
condition globally, the latter obtained by smoothly perturbing the second order coefficient $\s^2 s^2$ outside $D$. By performing a local analysis, we aim at exploiting this fact to prove error
estimates, uniform w.r.t. $x=(s,a)\in D$, for the intrinsic asymptotic expansions of the conditional expectation in
\eqref{eq:cond_expect}.% for $x=(s,a)\in D$.
%Nevertheless, if we fix a positive $\eps$, there exist infinitely many smooth functions defined on
%$\R^2$ which agree to $\frac{\s^{2}s^{2}}{2}$ on the half-plane $S_{\eps}:=\{s\geq \eps\}$ but are
%bounded below by a positive constant; let $h$ be such a function. Then the operator
%\begin{equation}
% h(s,a)\p_{ss}+s\p_{a}
% \end{equation}
%obviously coincides with the infinitesimal generator above on $S_{\eps}$ but also satisfies some non-degeneracy conditions we will specify later in Assumptions \ref{assA} and \ref{assB}.

%By exploiting the globality of this dummy operator we are able to prove error estimates for the asymptotic expansion of the conditional expectation
%\begin{equation}%\label{eq:cond_expect}
% u(t,s,a):=\Eb_{t,s,a} [\varphi(S_T,A_T)],\qquad  (t,s,a)\in[0,T]\times \R_{>0}\times\R_{>0}. %0\leq t<T,\ x\in D,
%\end{equation}
%\emph{as long as $(s,a)$ remains in % the non-degeneracy set
% $S_{\eps}$}.
%\end{example}
%an operator of the form \eqref{Ac}, that satisfies Assumptions  Clearly,
%more sophisticated multi-dimensional models, including stochastic and local volatility, as well as non-null interest rates, fall
%within the general framework above. %this framework.
%Adopting notations from \cite{FriedmanSDE},

\medskip %Coming back to the general case
In general, we assume that the infinitesimal generator of $X$ {\it coincides, on a
domain $D$} of $\R^{d}$, with a differential operator of the form }
\begin{align}\label{Ac}
 \Act &= \frac{1}{2}\sum_{i,j=1}^{p_0}\at_{ij}(t,x)\p_{x_{i}x_{j}}+ \sum_{i=1}^{p_0}\at_{i}(t,x)\p_{x_{i}}+ \langle B x, \nabla_x  \rangle %\sum_{i,j=1}^{d}b_{ij}x_{j}\p_{x_{i}} + \partial_t
 , \qquad (t,x)\in \R\times\Rd,
\end{align}
where $p_0\leq d$ and $\Act$ verifies the following
\begin{assumption}\label{assA}
$A_0:=\big(\at_{ij}(t,x)\big)_{i,j=1,\cdots,p_0}$ satisfies the non-degeneracy condition
\begin{align}\label{cond:parabolicity}
 \m M |\xi|^2< \sum_{i,j=1}^{p_0}\at_{ij}(t,x)\xi_{i}\xi_{j}< M |\xi|^2,\qquad (t,x)\in \R\times\Rd,\
 \xi\in\mathbb{R}^{p_0},
\end{align}
for some positive constants $M$ and $\m$;
\end{assumption}
\begin{assumption}\label{assB}
$B$ is a $(d\times d)$-matrix with constant entries satisfying the following structural condition
\begin{equation}\label{eq:B_blocks} B=\left(
\begin{array}{ccccc}
0 & 0 & \cdots & 0 & 0 \\ B_1 & 0 &\cdots& 0 & 0 \\ 0 & B_2 &\cdots& 0& 0 \\ \vdots & \vdots
&\ddots& \vdots&\vdots \\ 0 & 0 &\cdots& B_r& 0
\end{array}
\right)
\end{equation}
where each $B_j$ is a $(p_j\times p_{j-1})$-matrix of rank $p_j$ and
\begin{equation}
 p_0\geq p_1\geq \cdots \geq p_r\geq 1, \qquad \sum_{j=0}^r p_j = d.
\end{equation}
\end{assumption}
%\noindent \blu{Once again we would like to emphasize that the infinitesimal generator of $X$
%coincides with $\mathcal{A}$ in \eqref{Ac} only locally, on a suitable domain $D$. }

 Assumption \ref{assB} implies that vector
fields $\partial_{x_1},\dots,\p_{x_{p_{0}}}$ and
\begin{align}\label{eq:ste001}
 Y:=\langle B x, \nabla_x  \rangle+\partial_t
\end{align}
satisfy the H\"ormander condition (cf. \cite{LanconelliPolidoro1994}). %Note that, by H\"ormander's theorem, this implies
Under suitable regularity conditions %on the coefficients $\at_{ij}$, $\at_i$
that will be specified later, the ultra-parabolic operator
\begin{equation}\label{Kc}
 \Kc := \Act + \partial_t
\end{equation}
admits a fundamental solution (see \cite{Polidoro1994} and \cite{amrx}). In the case $p_0<d$,
which is the focus of this work, this is a remarkable fact as the second order part of $\Ac$ is
fully
degenerate at any point. Operators $\Kc$ of this kind %with $\Ac$ like in \eqref{eq:ste001}
are often referred to as \emph{Kolmogorov operators}.
%see in the sequel that Precise assumptions will be deferred until the next section, where we will impose precise assumptions on the matrix $B$ and on the matrix $A=\big(a\big)_{i,j=1,\cdots,p_0}$ guaranteeing that H\"ormander condition is satisfied.
%\bigskip\blu{Arrivato qui. *********************************************************}
%As a matter of example, consider the solutions of a stochastic differential equation
%\begin{equation}%\label{e2}
%  d Z_{t}=\mu(t,Z_t)d t+\s(t,Z_t) d W_{t},\qquad %\text{diag}(\underbrace{1,\dots,1}_{p_{0}},0,\dots,0)
%\end{equation}
%such that, on %a certain domain of $ \Rdd$
%$\R\times D$, the drift coefficient is given by $\mu(t,x) = B x + a(t,x)$, and the diffusion coefficient $\s(t,x)$ is such that
%%\begin{enumerate}
%%\item[(i)] $\mu(t,x) = B x$,
%%\item[(ii)] the diffusion coefficient $\s$ is such that
%\begin{equation}
%A:=\s \s^\top=\begin{pmatrix}
%    A_{0} & 0_{p_0\times (d-p_0)} \\
%    0_{(d-p_0)\times p_0} & 0_{(d-p_0)\times (d-p_0)} \
%  \end{pmatrix}.%\qquad A_0=\big(a\big)_{i,j=1,\cdots,p_0}.
%  \end{equation}
%%\end{enumerate}
%This class of processes is strictly related to that of the so-called averaged diffusions, which are stochastic systems where one or more components are defined in terms of the continuous average, arithmetic in this case, of other components.
%Consider for instance
%The main results of this paper are a novel analytical approximation and sharp asymptotic estimates
%for the conditional expectation

%for any given measurable function $\varphi$ such that the expected value above is well-defined.
Our analysis takes advantage of the intrinsic geometry and the related regularity structures
induced by the Kolmogorov operator $\Kc$. These features bring a number of benefits that are
explained here below, and distinguish our approach from others in the literature. %that do not take into account
%in particular, we aim at investigating %fully exploit
%the peculiar nature of the non-Euclidean geometric properties of
%this class of stochastic processes {\green (quest'ultima frase dopo i due punti mi sembra un po' troppo vaga)}.
It is worth to emphasize further that our results are carried out under strictly {\it local
assumptions} on the generator of $X$, which coincides with a Kolmogorov
operator on a domain $D$, not necessarily equal to $\R^{d}$. {This allows to include %in our analysis
degenerate models with relevant financial applications, such as the well-known CEV model (that is
when $\s$ in \eqref{BS} is not a constant but a function of $S$ of the form $\s(S)=S^{\g}$ for
some $\g\in\R$) and the Heston
stochastic volatility model as very particular cases.} %in particular, the H\"ormander condition for the fields $Y$
%and $X_i$, $i=1,\cdots,p_0$ and the \emph{intrinsic} regularity conditions on the coefficients
%$a_{ij}$, $a_i$ (see Assumption \ref{assC}) will be only required on $[0,T_0]\times D\subset
%\Rdd$, and no other assumption is required save that $Z$ is a Feller continuous diffusion.
The proof of our main result, Theorem \ref{th:error_estimates_taylor}, will be split in two
separate steps: first, in Theorem \ref{th:error_estimates_taylor_global}, we consider the case $D=\R^d$ %{\green (Theorem \ref{th:error_estimates_taylor_global})}
for which we employ some Gaussian upper bounds for the transition density of $X$; %fundamental solution of $\Kc$;
second, we adapt a localization procedure, originally introduced in \cite{Safonov1998} and lately
extended in \cite{Cinti2009135}, {which is based on the Gaussian bounds for a dummy %the density of
diffusion $\tilde{X}$ that is generated by $\mathcal{A}$ in \eqref{Ac}.} The latter localization
procedure is coherent with what is known in the theory of diffusions as the \emph{principle of not
feeling the boundary} (cf. \cite{Hsu}, \cite{Gathera2012}).

\subsection{Intrinsic Taylor-based asymptotic expansions}\label{sec:intro_expansions}
Intrinsic H\"older and Sobolev spaces for Kolmogorov operators were studied by several authors,
among others
%\cite{Ragusa},
\cite{Francesco}, \cite{BraCEMA}, \cite{Manfredini}, \cite{Lunardi1997},
%\cite{Kunze},
\cite{NyPasPol10} %, \cite{Priola}
and \cite{Menozzi}. {In this paper we use the intrinsic H\"older spaces $C^{n,\alpha}_B$ in
Definition \ref{def:C_alpha_spaces} below, as defined in \cite{PPP16} where the authors also
proved a Taylor formula with reminder expressed in terms of the homogeneous norm induced by the
operator (see Theorem \ref{eq:ste31} below).}  %of the group. (Questa frase mi piaceva piu prima. Se uno non e' esperto non sa cos'e' il gruppo.)
Deferring precise definitions and statements until Section \ref{sec:Kolmogorov}, the $n$-th order
{intrinsic} Taylor polynomial, centered at $\zeta=(s,\x) \in \R\times\Rd$, of a function $f\in
C^{n,\alpha}_B$ reads as
\begin{equation}\label{eq:tay_int}
 \TT_n(f,\z )(z):=  \sum_{{2 k + |\beta|_B \leq n}}\frac{1}{k!\,\beta!}%\bigg(  \prod_{i=0}^r \frac{1}{\big| \alpha^{[i]} \big|!}  \bigg)
 \big( Y^k \partial_{\xi}^{\beta}f(s,\xi)\big) (t-s)^k\big( x-e^{(t-s)B}\xi  \big)^{\beta},\qquad
 z=(t,x)\in\Rdd, %\label{eq:tay_int}
\end{equation}
where $|\beta|_B$, given in \eqref{hei}, is a suitable weight for the multi-index\footnote{We
denote by $\N$ the set of natural numbers and $\N_{0}=\N\cup\{0\}$.} $\b\in\N_{0}^{d}$. Such
Taylor expansion forms the cornerstone of the perturbation technique that we study in this paper.
Here below we summarize the intuitive idea behind it and its primary features.

We recall that, under mild assumptions that will be specified in Section \ref{sec:approximating1}, the function $u$ in \eqref{eq:cond_expect} satisfies %the backward Kolmogorov equation
\begin{equation}\label{equaz1}
  \begin{cases}
    \Kc u=0,\qquad &\text{on } [0,T[\times D, \\
    u(T,\cdot)=\varphi,\qquad &\text{on } D.
  \end{cases}
\end{equation}
Notice that \eqref{equaz1} is not a standard Cauchy-Dirichlet problem since no lateral
boundary conditions are imposed. In a series of papers, two of the authors propose a perturbative
method to carry out a closed-from approximation of solutions to \eqref{equaz1}
under the assumption that $\Kc$ in \eqref{Ac}-\eqref{Kc} is locally parabolic, i.e. $p_0=d$ and $B=0$ in \eqref{eq:B_blocks} %associated to a stochastic equation.
(for a recent and thorough description the reader can refer to
\cite{LPP4}, \cite{PP_compte_rendu}). %for the locally-parabolic case and to \cite{lorig-pagliarani-pascucci-1} for a more general setting including integro-differential equations.
The basic idea is to approximate the generator %$\Kc$
by Taylor expanding %(in the standard Euclidean sense)
its coefficients, and take advantage of some symmetry properties of Gaussian kernels. Sharp
short-time/small-noise asymptotic estimates for the remainder of the expansion are then proved.
%of the fact that, at order zero, the operator obtained by freezing the coefficients is a heat-type
%operator, whose fundamental solution is a Gaussian transition density. As such, this fundamental
%solution enjoys some symmetry property that were used to derive an explicit approximating
%sequence, and sharp short-time asymptotic error bounds were proved. In such parabolic framework,
%the regularity assumptions on the assumptions were rather classical, say $\at_{ij}, \at_i \in
%C^{n,\alpha}(\R_+\! \times D)$, in the classical sense.
In order to generalize the aforementioned technique to the case $p_0<d$, we perform an expansion
that is compatible with the sub-elliptic geometry induced by Kolmogorov operators. Assuming
$\at_{ij},a_{i}\in C_B^{N,1}$, we expand the operator $\Kc$ through the sequence
$\big(\Kc^{(\zb)}_n\big)_{0\le n\le N}$ defined as
\begin{equation}\label{e1_ter}
  \Kc^{(\zb)}_n=\frac{1}{2} \sum_{ {i,j=1}}^{p_{0}} \TT_n\left( \at_{ij},\bar{z}\right)(z) \p_{x_{i}x_{j}}+
  \sum_{ {i=1}}^{p_{0}} \TT_{n-1}\left( \at_{i},\bar{z}\right)(z) \p_{x_{i}}+ Y,%+\langle
%  Bx,\nabla\rangle+\p_{t},
 \qquad %\bar{z}=(\bar{t},\bar{x}),
 z=(t,x)\in\R\times\R^{d},
\end{equation}
where $\TT_n\left( \at_{ij},\bar{z}\right)$ is the Taylor polynomial of $\at_{ij}$, defined as in
\eqref{eq:tay_int}, centered at a fixed point $\bar{z}\in %[0,T]\times D
\Rdd$, and $\TT_{-1}\left(
\at_{i},\bar{z}\right)\equiv 0$.
\begin{remark}
When $p_0<d$, the intrinsic space $C^{n,\alpha}_B$ is strictly contained into the corresponding Euclidean H\"older space
$C^{n,\alpha}$: for this reason, the regularity assumptions on the coefficients are weaker than in the parabolic case.
\end{remark}

The leading term of the expansion, $\Kc^{(\zb)}_0$, is the Kolmogorov operator
with constant coefficients % obtained by freezing the coefficients of $\Kc$, i.e.
\begin{equation}\label{K0}
 \Kc^{(\zb)}_0=\frac{1}{2} \sum_{ {i,j=1}}^{p_{0}} \at_{ij}(\bar{z}) \p_{x_{i}x_{j}}%+\sum_{ {i=1}}^{p_{0}} \at_{i}(\bar{z}) \p_{x_{i}}
 +Y,%\frac{1}{2}\langle \nabla,A \nabla \rangle
 %+\langle Bx,\nabla\rangle+\p_{t},%\qquad \big(A_{ij}=a_{ij}(\bar{z})\big)_{1\leq i,j\leq p_0},\ (A_{ij}=0)_{p_0<i,j\leq d},
\end{equation}
defined on $\Rdd$. It is well-known that $\Kc^{(\zb)}_0$ admits a Gaussian fundamental solution %$\G_0(z,\zeta)$ %given by
%where
%\begin{align}\label{eq:G0_bis}
% \G_{0}( 0,0;t,x)=\frac{1}{\sqrt{(2\pi)^d|C(t)|}}\exp\Big(-\frac{1}{2}\langle C^{-1}(t)x,x\rangle \Big),\qquad  C(t)= \int_0^t e^{sB}A\, e^{sB^*}d s,
%\end{align}
%and that for $\Gamma_0$
that satisfies some remarkable symmetry properties
%\begin{equation}\label{eq:symmetry_Kol}
%(\xi- e^{(s-t)B}x)\,\G_0(t,x\,; s,\x)= C(s-t)\nabla_{x}\G_0(t,x\,; s,\x),
%\end{equation}
%which is
written  in terms of the increments appearing in the intrinsic Taylor polynomials in
\eqref{eq:tay_int}. %Upper bounds in terms of the intrinsic norm associated to $B$.
The main result of this paper, Theorem \ref{th:error_estimates_taylor}, %and %Theorem \ref{th:un_general_repres},
provides an explicit approximating expansion for $u(t,x)$ in \eqref{eq:cond_expect}, equipped with
sharp short-time error bounds, and can be roughly summarized as:
\begin{equation}\label{eq:asympt_exp}
 u(t,x)= u_0 (t,x) + \sum_{n=1}^{N} \Lc_n(t,T,x) \, u_0(t,x)+ %R_N(t,x;T),\qquad R_N(t,x;T)=
 \text{O}\left((T-t)^{\frac{ N+1+k }{2}}\right) \qquad \text{as }t\to T^{-},
\end{equation}
%as $t\to T^{-}$,
uniformly with respect to $x\in D$, where:
\begin{itemize}
  \item[-] the leading term $u_0$ is the solution of the Cauchy problem for $\Kc^{(\zb)}_{0}$ with final datum $\phi$;
  \item[-] $(\Lc_n)_{1\le n\le N}$ is a family of differential operators, acting on $x$, that can be explicitly computed in terms of the intrinsic Taylor polynomials
  $\TT_n\left(\at_{ij},\bar{z}\right)$ and $\TT_n\left( \at_{i},\bar{z}\right)$ (see Theorem \ref{th:un_general_repres});
  \item[-] the positive exponent $k$, contributing to the asymptotic rate of convergence, is the intrinsic H\"older exponent of $\varphi$. Precisely, $\phi\in C_B^{k}$ according to Definition \ref{Calpha} below.
%    depends on the intrinsic regularity of the terminal datum $\varphi$. Explicitly, it is the intrinsic H\"older order of $\phi\in C_B^{k}$. See Definition \ref{Calpha}.

\end{itemize}
Such approximation turns out to be optimal to several extents%[non so se è corretto] {\green (Si')}}
. In particular, the benefit in exploiting the intrinsic regularity is threefold:
%\begin{itemize}
%\item[i)]
first, since
%the operators $\Lc_n$ in \eqref{eq:asympt_exp} depend on the Taylor coefficients of the functions $a_{ij}$ and $a_i$, it is convenient to use the
the intrinsic Taylor polynomial %, being the latter
is typically a projection of the Euclidean one,
%{\blue (questo e' vero solo se il punto base del polinomio e' lo zero)} %In this way
we avoid taking up terms in the expansion that do not improve the quality of the approximation;
secondly, the fact that the increments of the intrinsic Taylor polynomial appear in the symmetries
of the fundamental solution of $\Kc_0^{(\zb)}$
allows to get compact approximation formulas; %which is a relevant point because approximations formulas are usually lengthy.
%,in order to get {\it compact approximation formulas}.
%This is a relevant point because approximations formulas are usually very lengthy;
{finally, the asymptotic rate of convergence of the expansion also depends on the {\it intrinsic
regularity} of the datum $\phi$, which is typically higher than the Euclidean regularity.
%since the bound in \eqref{eq:asympt_exp} also depends on the intrinsic regularity of the
%datum $\phi$, the asymptotic {\it rate of convergence} of this expansion is higher than its Euclidean counterparts.
This is particularly relevant in the financial applications (see %$\varphi_{\text{fixed}}$ in
%\eqref{eq:payoff_asian} and
Remark \ref{rem:asympt_conv_asian} below).
%. since the intrinsic regularity of $\phi$ is typically greater than the Euclidean regularity: this is the case, for instance, in the financial applications (see \eqref{payoff} {\blue below}).}
%\end{itemize}
\subsection{Applications to finance and comparison with the existing literature}\label{sec:asian}
The application of Kolmogorov operators %averaged-diffusion processes
in mathematical finance is particularly relevant in the pricing of Asian-style derivatives. These
are financial claims whose payoff is a function not only of the terminal value of an underlying
asset, but also of its average over a certain time-period. In most cases of interest, the problem
of computing the conditional expectation \eqref{eq:cond_expect}, which defines the no-arbitrage
price of such financial claims, is not known to have an explicit solution, and thus a considerably
large amount of literature has been developed in the last decades in order to find accurate and
quickly computable approximate solutions. Some of these approaches make use of asymptotic
techniques that lead to semi-closed approximation formulas. In this section we aim at firming our
results within the existing literature on analytical approximations of Asian-style derivatives.
Before to proceed we recall that other financial applications, where averaged-diffusion processes
are employed, include volatility models with path-dependent coefficients, e.g. the Hobson-Rogers
model \cite{HobsonRogers}.

Let us %consider again
resume our first example \eqref{BS} and now assume that $S$ follows the more general dynamics
\begin{equation}\label{eq:dyn_S}
 d S_t = \sigma \left(t, S_t,  A_t\right)d W_t.
\end{equation}
%To fix the ideas, consider the process $Z=(S,%V,
%A)$, whose dynamics are given by
%\begin{align}\label{eq:dyn_S}
%d S_t &= \sigma%_1
% \big(t, S_t, %V_t,
% A_t\big) d W%^{1}
%_t,\\
%%d V_t &= \mu_2\big(t, S_t, V_t, A_t\big) d t + \sigma_2 \big(t, S_t, V_t, A_t\big) d W^{2}_t,\\
%d A_t &= S_t d t.
%\end{align}
%%where $W^1$ and $W^2$ are two brownian motions that are assumed to be independent just to simplify the notation.
%Here, $S_t$ represents the risk neutral price of a certain risky asset (we assume for simplicity zero interest rate), and $A_t$ is its arithmetic continuous average. %, and $V_t$ represents an exogenous stochastic volatility factor of the market.
%Note that, even though this is not usual in the financial literature, here we let the coefficients of the model depend also on $A_t$, for our approach can naturally deal with this scenario in which the \emph{intrinsic} Taylor expansion actually differs from the Euclidean one. The process $Z$ as above fits the framework previously described with $d=2$, $p_0=1$ and $D = \R_{>0}\!\times\! \R_{>0}$; precisely, the generator $\Ac$ is as in \eqref{Ac} with
%\begin{equation}\label{eq:B_matrix_asian}
%\at_{11}(t,s,a) = \sigma^2(t,s,a),\qquad B = \begin{pmatrix}
%   0 & 0 \\
%   1 & 0
%  \end{pmatrix}.
%%\Kc = \frac{1}{2} \sigma^2 (t, s, a ) \partial_{ss} + s\, \partial_a + \partial_t.
%%\Kc = \frac{1}{2} \big( \sigma_1 (t, s, v ) \partial_{ss} + \sigma_2 (t, s, v ) \partial_{vv}  \big) + \mu_2 (t, s, v ) \partial_{v} + s \partial_a + \partial_t.
%\end{equation}
In this case, $\at_{11}(t,x_1,x_2) = \sigma^2(t,x_1,x_2)$ and its $n$-th order intrinsic Taylor polynomial
centered at $\z=(s,\xi_1,\xi_2)$ reads as
%\begin{equation}
%T_n a_{11}(\tau,\xi,\eta;t,s,a) = \sum_{2k + \beta_0 + 3 \beta_1  \leq n} \frac{1}{k!\,\beta_0!\, \beta_1!} \Big( (\partial_t +  \xi \partial_{\eta})^k \partial_{\xi}^{\beta_0}\partial_{\eta}^{\beta_1}a_{11}(s,\xi,\eta)\Big) (t-\tau)^k  (s - \xi)^{\beta_0} \big(a - \eta -(t-\tau) \xi \big)^{\beta_1} . %
%\end{equation}
\begin{equation}
 \TT_n \left(\at_{11},\z\right)(t,x_1,{x_2}) = \sum_{2k + \beta_0 + 3 \beta_1  \leq n} \frac{(\partial_{s} +  \xi_1 \partial_{\xi_2})^k \partial_{\xi_1}^{\beta_0}\partial_{\xi_2}^{\beta_1}\at_{11}(s,\xi_1,\xi_2)}
 {k!\,\beta_0!\, \beta_1!} (t-s)^k  (x_1 - \xi_1)^{\beta_0} \big({x_2} - \xi_2 -(t-s) \xi_1 \big)^{\beta_1} . %\big( x-e^{(t-s)B}\xi_1  \big)^{\beta}
\end{equation}
More explicitly, up to order $3$ we have
\begin{align}
 \TT_0 \left(\at_{11},\z\right)(t,x_1,{x_2}) &= \at_{11} (\z), \\
   \TT_1 \left(\at_{11},\z\right)(t,x_1,{x_2}) &= \TT_0 \left(\at_{11},\z\right)(t,x_1,{x_2}) +  (x_1 - \xi_1) \partial_{\xi_1} \at_{11} (\z),\\
 \TT_2 \left(\at_{11},\z\right)(t,x_1,{x_2}) &= \TT_1 \left(\at_{11},\z\right)(t,x_1,{x_2}) +  \frac{(x_1 - \xi_1)^2}{2!}\, \partial^2_{\xi_1} \at_{11} (\z) + (t-s)  (\partial_{s} +  \xi_1 \partial_{\xi_2}) \at_{11} (\z), \\
  \TT_3 \left(\at_{11},\z\right)(t,x_1,{x_2}) &= \TT_2 \left(\at_{11},\z\right)(t,x_1,{x_2}) +  \frac{(x_1 - \xi_1)^3}{3!}\, \partial^3_{\xi_1} \at_{11} (\z) +  ({x_2} - \xi_2 - (t-s)\xi_1 )\, \partial_{\xi_2} \at_{11} (\z)\\
  &\quad + (t-s) (x_1 - \xi_1)  (\partial_{s} +  \xi_1 \partial_{\xi_2}) \partial_{\xi_1} \at_{11} (\z),\hspace{-3pt}
\end{align}
which shows that the increment in the time variable appears only from the $2$nd order on, whereas the increment along the average variable appears from the $3$rd order on. As it was mentioned above, the operators $\Lc^{(\z)}_n$ appearing in the asymptotic expansion  in \eqref{eq:asympt_exp} can be explicitly computed by applying \eqref{eq:def_Ln}-\eqref{def_Gn}-\eqref{eq:M}-\eqref{eq:covariance_mean}. In this case they read as
\begin{align}
\Lc^{(\z)}_n(t,T,x) &=\frac{1}{2} \int_{t}^{T} \big(\TT_n  (a_{11},\z)  - \TT_{n-1}  (a_{11},\z) \big)   \big(s,\Mc^{(\z)}(s-t,{x_1},{x_2}) \big) \big( \partial_{x_1} - (s-t) \partial_{x_2} \big)^2 ds, \\ %\big(e^{-(r-t)B^*}\nabla_x\big)_1^2
\Mc^{(\z)}(t,{x_1},{x_2}) &= \Big( {x_1} + a_{11}(\z) t \partial_{x_1} - a_{11}(\z) \frac{t^2}{2} \partial_{x_2} \ , \ t x_1 + x_2 - a_{11}(\z)\frac{t^2}{2} \partial_{x_1} + a_{11}(\z)\frac{t^3}{6} \partial_{x_2} \Big).
\end{align}
In order to show an even more explicit sample, at order $1$ we have:
\begin{equation}
\Lc^{(\z)}_1(t,T,x)  = \frac{\partial_{\xi_1} a_{11}(\z)}{2}\int_t^T \Big(  ({x_1}-\xi_1) + {a_{11}(\z)(s-t) \partial_{x_1}} - \frac{a_{11}(\z)}{2}(s-t)^2 \partial_{x_2} \Big) \big( \partial_{x_1} - (s-t) \partial_{x_2} \big)^2 d s. %\\
%\Lc^{(\z)}_2(t,T,x) & = \frac{\partial_{\xi_1} a_{11}(\z)}{2}\bigg(  (T-t)(s-\xi_1) + \frac{(T-t)^2 \partial_s}{2} + \frac{(T-t)^3 \partial_a}{6}  \bigg)^2 .
\end{equation}

Two typical arithmetic Asian options are the so-called \emph{floating strike} and \emph{fixed
strike} Call options, whose payoffs are given respectively by
\begin{align}\label{eq:payoff_asian}
 \varphi_{\text{float}}(x_1,x_2) = \big(x_1 - x_2/T\big)^+, \qquad  \varphi_{\text{fixed}}(x_1,x_2) = \big( x_2/T -
 K\big)^+,
\end{align}
where $T$ is the maturity and $K$ is the strike price.
%Analogous payoffs can be of course considered for Asian Put options.
\begin{remark}\label{rem:asympt_conv_asian}
The payoff $\varphi_{\text{fixed}}$ is Lipschitz continuous in the standard Euclidean sense but
has higher intrinsic regularity (namely, $C^{3}_{B}$ according to Definition \ref{Calpha}, see
also Example \ref{exreg}): this property reflects a higher rate of convergence of the asymptotic
expansion \eqref{eq:asympt_exp} compared with other expansions based on standard Taylor
polynomials. {On the other hand, because of its explicit dependence on $x_{1}$, the payoff
$\varphi_{\text{float}}$ is only $C^{1}_{B,\text{\rm loc}}$.}
%\begin{equation}\label{eq:reminder_asian}
%R_N(t,x;T)= \text{O}\left(\big( \| \s \|_{\infty}^2 (T-t)\big)^{\frac{ N+3%_{\phi}
% }{2}}\right),\qquad \text{as } \| \s \|_{\infty}^2 (T-t)\to 0.
%\end{equation}
\end{remark}
Even in the simplest case of constant volatility, %$\sigma(t,x_1,x_2)=\sigma s$,
i.e. in the Black\&Scholes model, both the marginal distribution of $A_t$ and the joint
distribution of $(S_t,A_t)$ are difficult to characterize analytically. The distribution of $A_t$
was given an integral representation in the pioneering work \cite{Yor1992a}, though that result is
of limited practical use in the valuation of Asian options. The approximation formulas that we
propose in this paper were applied heuristically in \cite{FPP}, where intensive numerical tests
were performed to confirm their accuracy. However, the general hypoelliptic framework that we
consider here clearly allows for several generalization, including more general dynamics and more
sophisticated Asian style-derivatives including stochastic local volatility models such as the CEV
and the Heston models \cite{heston1993}. An interesting example is also given by a generalized
type of Asian option, where the average is weighted w.r.t. the volume of traded assets: these
options are written on the \emph{Volume Weighted Average Price} (VWAP), a trading benchmark used
especially in pension plans (see, for instance, \cite{Novikov}). The dynamics of the traded volume
$V$ are lead by an additional stochastic factor that has to be chosen as to reflect the
corresponding volume statistics, and the average process $A$ is then given by
\begin{equation}
A_t =\frac{ \int_{0}^t S_{\tau} V_{\tau} d \tau} {\int_{0}^t V_{\tau} d \tau}.
\end{equation}
%For instance, one could:
%\begin{itemize}
%\item introduce, in the dynamics of $S$, an exogenous stochastic volatility factors $V$ driven by an additional Brownian motion. In other words, by replacing $\sigma(t,S_t,A_t)$ with $\sigma(t,S_t,V_t,A_t)$ in \eqref{eq:dyn_S}, one would consider a so-called \emph{local-stochastic volatility} model, as opposed to a mere \emph{local volatility} model;
%\item study a generalized type of Asian option, where the average is weighted w.r.t. the volume of traded assets.
%Precisely, these are options written on the \emph{Volume Weighted Average Price (VWAP)}, a trading
%benchmark used especially in pension plans. The dynamics of the traded volume $V$ are lead by an
%additional stochastic factor that has to be chosen as to reflect the corresponding volume
%statistics, and the average process $A$ is then given by
%\begin{equation}
%A_t =\frac{ \int_{0}^t S_{\tau} V_{\tau} d \tau} {\int_{0}^t V_{\tau} d \tau}.
%\end{equation}
%\end{itemize}

%\subsection{Comparison with the existing literature}\label{sec:compar}

As it was previously argued, our technique makes use of the {intrinsic} Taylor polynomials in
\eqref{eq:tay_int} in order to be consistent with the subelliptic geometry induced by Kolmogorov
operators. This differentiates our approach from others appearing in the literature that are based
on classical Euclidean expansions. In the relevant paper \cite{Gobet2014_weak}, Malliavin calculus
techniques were employed to derive analytical approximations for the law of a general averaged
diffusion. When applied to the pricing of arithmetic Asian options, the approach in
\cite{Gobet2014_weak} returns an expansion whose leading term is the price of a geometric Asian
option. Correcting terms are computed by Taylor expanding the coefficients of the diffusion and
error estimates depend on standard Euclidean
regularity assumptions on the coefficients and on the payoff function. %Although this technique is general enough to deal
%with a wide class of averages, the error analysis relies on some Euclidean regularity assumptions
%on the coefficients and on the payoff function.
In %the papers \cite{Bouaziz},
\cite{TsaoLin} and \cite{Wojakowski}, the authors followed a different approach and carried out
a Taylor based-expansion of the joint distribution $(S_t,A_t)$ to analytically approximate
the price of an Asian option (possibly, {forward-starting}); this technique seems to be limited to
the Black\&Scholes dynamics. %for the underlying asset.
%Compared to all these methods, the asymptotic rate of convergence of the \emph{intrinsic} Taylor-based expansion \eqref{eq:asympt_exp} %that we propose
%is higher whenever the intrinsic regularity of the payoff function is higher than the Euclidean
%one. This is the case, for instance, of an Asian option with fixed-strike (see Remark
%\ref{rem:asympt_conv_asian}).
Other approximations, based on Taylor expansions and on Watanabe's
theory, can be found in \cite{KunitomoTakahashi1992}, %\cite{ShirayaTakahashi2010},
%\cite{ShirayaTakahashiToda2009},
though no rigorous error bounds are provided.

For sake of completeness, we also give a brief, and by no means exhaustive, overview of the
existing literature concerning other approaches to the pricing of Asian options.
%{\bf Laplace transform approach:}
Within the Black\&Scholes framework, \cite{GemanYor1992} derived an analytical expression for the
Laplace transform of $A_t$. However, several authors %(cf. %\cite{Shaw98}, \cite{FuMadanWang1998}, \cite{Dufresne2002})
pointed out some stability issues related to the numerical inversion of the Laplace transform,
which lacks accuracy and efficiency in regimes of small volatility or short time-to-maturity. This
is also a disadvantage of the Laguerre expansion proposed in \cite{Dufresne2000}. \cite{Shaw2003}
used a contour integral approach based on Mellin transforms to improve the accuracy of the results
in the case of low volatilities, albeit at a higher computational cost. As opposed to numerical
inversion, \cite{Linetsky2004} derived an eigenfunction expansion of the transition density of
$A_t$ (see also \cite{Donati-Martin-2001}) by employing spectral theory of singular
Sturm-Liouville operators. Although it returns in general very accurate results, Linetsky's series
formula may converge slowly in the case of low volatility and become computationally expensive.
Note that, by opposite, the analytical pricing formulas we propose here do not suffer any lack of
accuracy or efficiency in these limiting cases. In actual fact, Theorem
\ref{th:error_estimates_taylor} and Remark \ref{rem:small_vol} show that the accuracy improves as volatility and/or time to
maturity get smaller. Again in the particular case of the Black\&Scholes model, and for special
homogeneous payoff functions, it is possible to reduce the pricing PDE in \eqref{equaz1} to a one
state variable PDE. PDE reduction techniques were initiated in \cite{Ingersoll} and applied to the
problem of pricing Asian options by several authors, including
\cite{RogersShi1995,Vecer2001} and \cite{DewynneShaw2008}. %\cite{Zhang2001},%\cite{Vecer2002},
Eventually, other approaches include the
parametrix expansion in \cite{CorielliFoschiPascucci2010} and the moment-matching techniques in
\cite{Dufresne2001CEV,Deelstra,FusaiTagliani} and \cite{FordeJacquier2010} among others.

  \section{ Kolmogorov operators and intrinsic H\"older spaces}\label{sec:Kolmogorov}
%%%%%%%%%%%%%%%%%%%%%%%%%%%
In this section we collect some known facts about the intrinsic geometry of Kolmogorov operators.
We also recall the definition of intrinsic H\"older spaces and the Taylor formula recently proved in
\cite{PPP16}.
We consider the prototype Kolmogorov operator obtained by \eqref{Ac}-\eqref{Kc} with $A_{0}$ equal to a scalar $(p_{0}\times p_{0})$-matrix
 %-identity matrix
and $a_{i}\equiv 0$, $i=1,\dots,p_{0}$, i.e. %Thus we have %{\green (io userei una notazione diversa da $\Kc_0$)}
\begin{align}\label{eq:heatoperator}
 \Kc^{\Lambda}%_{0}
 := \frac{\Lambda}{2}\sum_{i=1}^{p_0}\p^2_{x_{i}}+ \langle B x, \nabla_x  \rangle + \partial_t %\sum_{i,j=1}^{d}b_{ij}x_{j}\p_{x_{i}} + \partial_t
 , \qquad (t,x)\in  \R\times\mathbb{R}^d, \qquad \Lambda >0.
\end{align}
In this case we say that $\Kc^{\Lambda}$ is a {\it constant coefficients Kolmogorov operator}.
%\begin{remark}\label{rem:hormander_condition}
By Assumption \ref{assB}, the vector fields $\p_{x_{1}},\dots,\p_{x_{p_{0}}}$ and $Y$ in
\eqref{eq:ste001} satisfy the H\"ormander's condition and therefore $\Kc^{\Lambda}$ is hypoelliptic.
%More explicitly the $d+1$ dimensional Lie algebra is generated by the vector fields
%\begin{equation}
%\label{eq:def_vector_fields}
%X_1\equiv \p_{x_1},\: X_2\equiv \p_{x_2},\quad \dots \quad X_{p_0}\equiv \p_{x_{p_0}},\quad Y\equiv \langle Bx,\nabla\rangle +\p_t.
%\end{equation}
%\end{remark}
%{\green The interest in recalling the properties and the regularity structures induced by this operator stems from the fact that $\Kc_0$ plays a key rule in the construction of the fundamental solution for the variable-coefficients Kolmogorov operator $\Kc$}.
As it was first observed in \cite{LanconelliPolidoro1994}, $\Kc^{\Lambda}$ has remarkable invariance
properties with respect to the homogeneous Lie group %(in the sense of \cite{FollandStein1982})
$\mathcal{G}_B=\left(\Rdd,\circ,\left(D(\lambda)\right)_{\l>0}\right)$ where ``$\circ$" is the group law
defined as
\begin{equation}\label{eq:translation}
 (t,x)\circ (s ,\xi) = \left(t+s ,{e^{s  B}}x+\xi\right),\qquad (t,x),(s ,\xi)\in \Rdd,
\end{equation}
and $\left(D(\lambda)\right)_{\l>0}$ are the dilations given by
\begin{equation}\label{eq:dilation}
 D(\lambda)=\textrm{diag}\big(\lambda^2,\lambda I_{p_0},\lambda^{3}I_{p_1},\dots,\lambda^{2r+1}I_{p_r}\big),%(t,x), \qquad (t,x)\in\Rdd,
\end{equation}
where $I_{p_j}$ denote the $(p_j\times p_j)$-identity matrices. Precisely, it was proved in
\cite{LanconelliPolidoro1994} that $\Kc^{\Lambda}$ is invariant with respect to the left
$\circ$-translations and homogeneous of degree two with respect to the dilations
$\left(D(\lambda)\right)_{\l>0}$. Notice that $\mathcal{G}_B$ is completely determined by the
matrix $B$; moreover, the identity element in $\mathcal{G}_B$ is ${\text{Id}=(0,0)}$ and the inverse is
$(t,x)^{-1}=\left(-t,{-e^{-tB}}x\right)$. For convenience, we also denote by
\begin{equation}\label{eq:dilation_zero}
 D_{0}(\lambda)=\textrm{diag}\big(\lambda I_{p_0},\lambda^{3}I_{p_1},\dots,\lambda^{2r+1}I_{p_r}\big),%(t,x), \qquad (t,x)\in\Rdd,
\end{equation}
the ``spatial part'' of $D(\l)$.
%\begin{remark}
%\label{rem:homogeneity}
% %$$\mathcal{G}_B:=\left(\Rdd,\circ,D(\lambda)\right).$$
%%in Section \ref{sec:intrinsic_geom}
%\end{remark}
A homogeneous norm on $\mathcal{G}_B$ is defined as follows:
\begin{equation}\label{eq:def_norm}
 \norm{(t,x)}_B=|t|^{1/2}+[x]_B,\qquad [x]_B:=\sum_{j=1}^d |x_j|^{1/\s_j},{\qquad (t,x)\in\R\times\R^d,}
\end{equation}
where $(\s_j)_{1\leq j\leq d}$ are the integers such that
\begin{equation}
\label{eq:dilation_exponents}
 D_{0}(\lambda)=\textrm{diag}\big(\lambda^{\s_1},\dots,\lambda^{\s_d} \big),
\end{equation}
{that is $\s_1=\cdots = \s_{p_0}=1$, $\s_{p_0+1}=\cdots =\s_{p_0+p_1}=3$ and so forth.}

In the general setting of homogeneous Lie groups, H\"older spaces and intrinsic Taylor polynomials
can be defined as in \cite{FollandStein1982} and \cite{Bonfiglioli2009}. For the particular case
of homogeneous Lie groups induced by Kolmogorov operators,  \cite{PPP16} provides a deeper
analysis of the intrinsic Taylor polynomials
%defined in terms of the pseudo-distance
%\begin{equation}\label{eq:def_dist}
%  d_B(z,\z):=\norm{\z^{-1}\circ z}_B , \qquad z,\z \in\Rdd,
%\end{equation}
under optimal regularity assumptions.
%\blu{(see, for instance, Section 2 in \cite{Francesco} for a
%description of the basic properties of $d_{B}$)}.

For any Lipschitz vector field $Z$ on $\Rdd$, we denote by $\d\mapsto e^{\d Z}(z)$ the integral
curve of $Z$ starting from $z$: in particular, we have
\begin{equation}\label{eq:def_curva_integrale_campo}
 e^{\d \p_{x_{i}} }(t,x)=(t,x+\delta \mathbf{e}_i),\quad i=1,\cdots,p_0,\qquad
 e^{\d Y }(t,x)=(t+\delta,e^{\delta B}x), %, \qquad z=(t,x)\in\Rdd.
\end{equation}
where $\mathbf{e}_i$ denotes the $i$-th element of the natural Euclidean basis of $\R^{d}$. We say
that a function $u$ is \emph{$Z$-differentiable} at $z$ if $\d\mapsto u\left(e^{\d Z }(z)\right)$
is differentiable at $0$ and in that case $\frac{d}{d \d} u\left(e^{\d Z }(z)\right)\big|_{\d=0}$
is referred to as the \emph{Lie derivative of $u$ at $z$ along $Z$}. Since the vector fields
$\p_{x_{1}},\dots,\p_{x_{p_{0}}}$ and $Y$ are $D(\l)$-homogeneous of degree one and two
respectively, we associate to $\p_{x_{1}},\dots,\p_{x_{p_{0}}}$ and $Y$ the {\it formal degrees}
one and two respectively. %\blu{In general, if a Lipschitz vector field $Z$ has formal degree
%$\formaldeg_{Z}\in\R_{>0}$ then for any $\a\in\,]0,\formaldeg_{Z}]$ we say that $u\in
%C_{Z}^{\alpha}$ if the semi-norm
%\begin{equation}
% \norm{u}_{C^{\a}_{Z}}:=\sup_{\Rdd}|u|+
% \sup_{z\in\Rdd\atop \d\in \R\setminus\{0\}} \frac{
% \left|u\left(e^{\delta Z }(z)\right)-
% u(z)\right|}{|\delta|^{\frac{\alpha}{\formaldeg_{Z}}}}
%\end{equation}
%is finite.}
In general, if a Lipschitz vector field $Z$ has formal degree $\formaldeg_{Z}>0$ and
$u$ is a function on $\Rdd$, then for any $\a\in\,]0,\formaldeg_{Z}]$ we say that $u\in
C_{Z}^{\alpha}\equiv C_{Z}^{\alpha}(\Rdd)$ if the norm
\begin{equation}
 \norm{u}_{C^{\a}_{Z}}:=\sup_{\Rdd}|u|+
 \sup_{\d\in\R\setminus\{0\}\atop z\in\Rdd%\atop e^{\d Z}(z)\in]0,T[\times\Rd
 } \frac{
 \left|u\left(e^{\delta Z }(z)\right)-
 u(z)\right|}{|\delta|^{\frac{\alpha}{\formaldeg_{Z}}}}
\end{equation}
is finite. %{\green (così' pero' il teorema di Taylor non vale. Perché' definiamo gli spazi sulla striscia se i coefficienti di $\Kc$ sono definiti sul $\Rdd$?)}
Now we define the intrinsic H\"older spaces on the homogeneous group $\mathcal{G}_B$.
\begin{definition}\label{def:C_alpha_spaces}
Let $\a\in\,]0,1]$ and $n\in\Nb$ with $n\ge2$, then:
\begin{itemize}
  \item [i)] $u\in C^{0,\a}_{B}$ if $u\in C^{\a}_{Y}$ and $u\in C^{\a}_{\p_{x_{i}}}$ for any
  $i=1,\dots,p_{0}$;
  \item [ii)] $u\in C^{1,\a}_{B}$ if $u\in%C^{0,\a}_{B}\cap
  C^{1+\a}_{Y}$ and $\p_{x_{i}}u\in C^{0,\a}_{B}$ for any
  $i=1,\dots,p_{0}$;
  \item [iii)] $u\in C^{n,\a}_{B}$ if %$u\in C^{n-1,\a}_{B}$,
  $Yu\in C^{n-2,\a}_{B}$ and $\p_{x_{i}}u\in
  C^{n-1,\a}_{B}$ for any $i=1,\dots,p_{0}$.
\end{itemize}
We also introduce the norms:
\begin{align}\label{e9}
  \norm{u}_{C^{0,\a}_{B}}&:=\norm{u}_{C^{\a}_{Y}}+\sum_{i=1}^{p_0}\norm{u}_{C^{\a}_{\partial_{x_i}}},\\ \label{e100}
  \norm{u}_{C^{1,\a}_{B}}&:=%\blu{\norm{u}_{C^{0,\a}_{B}}+}
  \norm{u}_{C^{1+\a}_{Y}}+\sum_{i=1}^{p_0}\norm{\partial_{x_i}u}_{C^{0,\a}_{B}},\\ \label{e11}
  \norm{u}_{C^{n,\a}_{B}}&:=%\blu{\norm{u}_{C^{n-1,\a}_{B}}+}
  \norm{Y u}_{C^{n-2,\a}_{B}}+\sum_{i=1}^{p_0} \norm{\partial_{x_i}u}_{C^{n-1,\a}_{B}}.
\end{align}
\end{definition}
\begin{remark}\label{rem:inclusions}
%{\bf [Mi sembra completamente ovvio dalla definizione.]}
%The following inclusion holds:
%\begin{equation}
% C^{k+1,\a}_{B,\text{\rm loc}} \subseteq C^{k,\a}_{B,\text{\rm loc}},\qquad k\geq 0.
Notice that $C^{n+1,\a}_{B} \subseteq C^{n,\a}_{B}$ for any $n\in\N_{0}$.
% and $0<\a' \leq \a \leq 1$.
\end{remark}
For any multi-index $\beta=(\beta_1,\cdots, \beta_d)\in \Ndzero$, % will denote a multi-index. As
we define the $B$-length of $\b$ as
\begin{equation}\label{hei}
 |\beta|_B:=\sum_{j=1}^d \s_{j}\b_{j},
\end{equation}
with $\s_{j}$ as in \eqref{eq:dilation_exponents}. We are now in position to state the intrinsic Taylor
theorem that was proved in \cite{PPP16}.
\begin{theorem}\label{th:main}
Let  $\alpha\in\,]0,1]$ and $n\in\Nzero$. If $u\in C^{n,\a}_{B}$ then the derivatives
\begin{align}\label{eq:ste31}
 Y^k \partial_x^{\beta}u\in C^{n-2k-|\beta|_B,\alpha}_B\qquad \text{ for }\ {0\leq 2 k + |\beta|_B \leq n},
\end{align}
exist and therefore, for any point $\z=(s,\xi)$, the \emph{$n$-th order $B$-Taylor polynomial}
$\TT_{n}(u,\z)(\cdot)$ in \eqref{eq:tay_int} is well defined. Moreover, we have
\begin{equation}\label{eq:estim_tay_n}
  \left|u(z)-\TT_{n}(u,\z)(z)\right|\le c_B \|u\|_{C^{n,\a}_{B}}  \|\z^{-1}\circ z\|_{B}^{n+\a}, \qquad {z,\z\in\Rdd,}%z=(t,x),\z={\blue (s,\xi)}\in\Rdd,
\end{equation}
where $c_B$ is a positive constant that only depends on $B$.
\end{theorem}
Definition \ref{def:C_alpha_spaces} and Theorem \ref{th:main} will be used in the next section,
respectively, to specify suitable regularity conditions on the coefficients of $\Kc$ in
\eqref{Ac}-\eqref{Kc}, and to expand them as in \eqref{e1_ter}. However, as anticipated in Section
\ref{sec:intro_expansions}, the intrinsic regularity of the terminal datum $\phi$ plays as well a
key role in the error analysis of the expansion \eqref{eq:asympt_exp}. This motivates the
following
\begin{definition}\label{Calpha}
Let $k\in\,]0,2r+1]$. We denote by $C_{B}^{k}(\R^{d})$ the space of functions $\phi$ on $\R^{d}$
such that
\begin{equation}%\label{}
 \left|\phi(x)-\phi(y)\right|\le C[x-y]_{B}^{k},\qquad x,y\in\R^{d},
\end{equation}
for some positive constant $C$, where $[\cdot]_{B}$ is the norm on $\R^{d}$ defined in
\eqref{eq:def_norm}. We also set
\begin{equation}\label{normC}
  \|\phi\|_{C^{k}_{B}(\R^{d})}=\sup_{x\neq y}\frac{\left|\phi(x)-\phi(y)\right|}{[x-y]_{B}^{k}}.
\end{equation}
Moreover, by convention, $C^{0}_{B}(\R^{d})$ is the set of bounded and continuous functions on
$\R^{d}$ and $\|\phi\|_{C^{0}_{B}(\R^{d})}=\|\phi\|_{L^{\infty}(\R^{d})}$.
\end{definition}
\begin{example}\label{exreg}
Consider the case of arithmetic Asian options with fixed strike discussed in Section
\ref{sec:asian}, i.e.
\begin{equation}
B = \begin{pmatrix}
   0 & 0 \\
   1 & 0
  \end{pmatrix},\qquad \varphi_{\text{fixed}}(x_{1},x_{2}) = \left(x_{2}/T - K\right)^+.
\end{equation}
According to Definition \ref{Calpha}, $\phi_{\text{fixed}} \in C^{3}_{B}(\R^{2})$ even if it is
only Lipschitz continuous in the Euclidean sense.
\end{example}
\section{Approximate solutions and error bounds}\label{sec:approximating1}%\setcounter{equation}{0}
Let $X$ be a Feller process as defined in the introduction: in particular, we assume that the
infinitesimal generator of $X$ coincides with operator $\Ac$ in \eqref{Ac} on a fixed domain $D$
of $\R^{d}$. Moreover, $\Ac$ satisfies Assumptions \ref{assA} and \ref{assB}. Throughout this
section $N\in \N_0$ and $T>0$ are fixed and we also require the following assumptions to be in
force:
\begin{assumption}\label{assC}
%$Z$ is a \emph{Feller process on $D$}, i.e. for any $T\in\,]0,T_{0}[$ and {$\phi\in C_0(\R^{d})$}
%the function $(t,z)\mapsto %(\TT_{t,T}\phi)(z)
%E_{t,z}\left[\phi(Z_{T})\right]$ is continuous on $[0,T[\times D$. Moreover, $Z$ is a \emph{local} diffusion on $[0,T_0[\times D$ whose generator coincides on this set with $\Act_t$, %on
%%and its infinitesimal generator $\Ac_t$ is such that
%%\begin{equation}
%%\Ac_t \varphi =  \Act_t \varphi , \qquad t\in [0,T_0[,\quad \varphi\in C^2_0(D),
%%\end{equation}
%the latter being a second order differential operator defined on $\Rd$ of the type \eqref{Ac} %, with the coefficients $\at_{i},\at_{ij}$
%satisfying, for some $M>0$, $\eps\in]0,1[$, the following two conditions:
%\begin{enumerate}
%\item[i)] {\it Uniform ellipticity:} %there exist $M>0$, $\eps\in]0,1[$, such that
%\begin{align}\label{cond:parabolicity}
%\eps M |\xi|^2\leq \sum_{i,j=1}^{p_0}\at_{ij}(t,x)\xi_{i}\xi_{j}\leq M |\xi|^2,\qquad
%t\in\left[0,{T}_0\right],\ {\blue x\in\Rd, \ }\xi\in\mathbb{R}^{p_0}.
%\end{align}
%\item[ii)] {\it Regularity and boundedness:} \
The coefficients $\at_{ij},\at_{i}$ of $\Ac$ belong to $\in C_{B}^{N,1}$ and
\begin{equation}\label{condM}
  \norm{a_{ij}}_{C^{N,1}_{B}}, \norm{a_{i}}_{C^{N,1}_{B}}\leq M,
\end{equation}
with $M$ as in \eqref{cond:parabolicity}. %{\green (anche qui, i coefficienti sono funzioni su $\Rdd$ ma gli spazi H\"older sono sulla striscia.)}
\end{assumption}
\begin{assumption}\label{assD}
The final datum $\phi$ is a continuous function with sub-exponential growth such that $u=u(t,x)$ in \eqref{eq:cond_expect} is
well defined and belongs to $L^{\infty}([0,T]\times D)$. Moreover, there exists $\psi\in
C_{B}^{k}(\R^{d})$, with $k\in[0,2r+1]$, such that $\phi=\psi$ on $D$.
\end{assumption}
\noindent %Notice that \eqref{condM} is not a restrictive condition.
%For a given payoff function $\phi\in C(\Rdd)$ we are interested in computing the expected value $u(t,z)$ as in \eqref{eq:cond_expect}.
The following preliminary result can be proved as in \cite{JansonTysk2006} or \cite{PP15}, using the Schauder estimates and the results on Green functions proved in \cite{Francesco}. % . For
%completeness, we provide a sketch of the proof in the Appendix.
\begin{proposition}\label{prop:solution_cauchy_pb}
{Let Assumptions \ref{assA}, \ref{assB}, \ref{assC} and  \ref{assD} be in force. %Let $u$ as in
%\eqref{eq:cond_expect} with $\varphi\in C_{B}^{k}(\Rd)$, $k\in[0,2r+1]$.
Then, $u \in C([0,T]\times D) \cap C^{N+2,1}_{B,\text{\rm loc}}$ and satisfies \eqref{equaz1}.} %{\green (gli spazi loc?)}
\end{proposition}
As was mentioned in the introduction, the idea behind our approximation of $u=u(t,x)$ in
\eqref{eq:cond_expect} is to expand the generator of $X$ by approximating the coefficients
$a_{ij}$ and $a_j$ in \eqref{Ac} by
means of their \emph{intrinsic} Taylor polynomials in \eqref{eq:tay_int}. % Theorem \ref{th:main}.
Thus we fix $\bar{z}=(\bar{t},\bar{x})\in \Rdd$ and consider the sequence
$\big(\Kc^{(\zb)}_n\big)_{0\le n\le N}$ in \eqref{e1_ter}.
We recall that, by Assumptions \ref{assA} and \ref{assB}, %the operator $\Kc^{(\bar{z})}_{{\blue n}}$ is well defined. %s $\Kc^{(\bar{z})}_n,1\leq n\leq N,$ are well defined.
%Furthermore, as it was pointed in Remark \ref{rem:hormander_condition}, Assumption \ref{assB}
%ensures that the H\"ormander's condition holds for the Kolmogorov operator
%\begin{equation}
% \Kc^{(\bar{z})}_0 :=  \frac{1}{2}\sum_{i,j=1}^{p_0}  a_{ij}(\bar{z}) \partial_{x_i x_j} + \langle Bx,\nabla\rangle + \partial_t, \qquad x\in\Rd.
%\end{equation}
$\Kc^{(\bar{z})}_0$ in \eqref{K0} has a fundamental solution $\Gamma^{(\bar{z})}_{0}$ that is the
$d$-dimensional Gaussian density
\begin{equation}\label{Gamma0}
 \Gamma^{(\bar{z})}_{0}(t,x;T,y) =  \frac{1}{  \sqrt{(2\pi)^{d}|\Cv_{\bar{z}}(T-t)|} }
    \exp\left(-\frac{1}{2}\langle\Cv_{\bar{z}}^{-1}(T-t) (y - e^{(T-t)B}x)),
    (y -e^{(T-t)B}x)\rangle\right)
\end{equation}
with covariance matrix $\Cv_{\bar{z}}(t)$ %and mean vector $x+\mv(t,T)$
given by
\begin{align}\label{eq:covariance_mean}
 \Cv_{\bar{z}}(t)=  \int_0^t e^{sB}A(\bar{z})e^{sB^*} d s,\qquad %\begin{equation}
 A(\zb):=\begin{pmatrix}
    A_{0}(\zb) & 0_{p_0\times (d-p_0)} \\
    0_{(d-p_0)\times p_0} & 0_{(d-p_0)\times (d-p_0)}.
  \end{pmatrix}.%\qquad A_0=\big(a\big)_{i,j=1,\cdots,p_0}.
%  \end{equation}
\end{align}
Next we formally expand the expected value $u$ in \eqref{eq:cond_expect} as
\begin{align}
  u  \ \approx\ U^{(\bar{z})}_N := \sum_{n=0}^N u^{(\bar{z})}_n. \label{eq:v.expand}
\end{align}
Inserting \eqref{e1_ter}, \eqref{eq:v.expand} into \eqref{equaz1} and formally collecting terms of the same
order, we find that the functions $u^{(\bar{z})}_{n}$ satisfy the following sequence of nested
Cauchy problems
\begin{align}\label{eq:v.0.pide}
 &\begin{cases}
 \Kc^{(\bar{z})}_0 u^{(\bar{z})}_0 =  0,\qquad & \text{on } [0,T[\times \Rd  % t\in[0,T[,\ x\in\mathbb{R}^d
 , \\
 u^{(\bar{z})}_0(T,\cdot) =  \phi,&  \text{on }\mathbb{R}^d,
\end{cases}
\intertext{and} \label{eq:v.n.pide}
 &\begin{cases}
  \Kc^{(\bar{z})}_0 u^{(\bar{z})}_n  =  - \sum\limits_{h=1}^{n} \big(\Kc^{(\bar{z})}_h - \Kc^{(\bar{z})}_{h-1} \big) u^{(\bar{z})}_{n-h},
  \qquad & \text{on } [0,T[\times \Rd  % t\in[0,T[,\ x\in\mathbb{R}^d
  , \\
 u^{(\bar{z})}_n(T,\cdot) =  0, &   \text{on }\mathbb{R}^d . % x \in\mathbb{R}^d.
\end{cases}
\end{align}
The explicit representation of the terms $u_{n}^{(\zb)}$ of the expansion is given in Theorem
\ref{th:un_general_repres}.
\begin{remark}
In the above construction, the approximation in \eqref{eq:v.expand} is defined in terms of a
sequence of Cauchy problems that admit a unique non-rapidly increasing solution.
%determine a unique approximating sequence. %On the other hand, the expected value $u $ as defined in
%\eqref{eq:cond_expect} is only \emph{one} solution of the Cauchy problem \eqref{equaz1}, which is
%given on a local domain $D\subset \Rd$ and as such has infinite solutions unless to fix
%appropriate boundary conditions.
Conversely, equations \eqref{equaz1} do not have a unique solution unless additional lateral boundary
conditions are posed. Nevertheless, Theorem \ref{th:error_estimates_taylor} below states that the
above expansion is asymptotically convergent in the limit of short-time%and/or small-noise
,
uniformly on compact subsets of $D$.  %Although this might be seen as surprising at first glance,
This is in line with the so-called \emph{principle of not feeling the boundary} (cf. \cite{Hsu},
\cite{Gathera2012}). Basically, the same asymptotic result would hold for any bounded solution of
equations \eqref{equaz1}, with error bounds depending on the $L^{\infty}$-norm of the solution. Of
course, knowing the boundary conditions %, i.e. the value of $u(t,x)$ for $x\in \partial D$,
would allow to construct an approximate sequence that is also accurate near the boundary; this is the case of barrier options in the financial applications.
\end{remark}
The choice of the basis point $\bar{z}$ is somewhat arbitrary, but only some particular choices
allow for performing a rigorous error analysis. For instance, here below we consider the case
$\bar{z}=z=(t,x)$. However, although we omit to write separate proofs, the same results hold by
setting $\bar{z}=(T,x)$. In the following statement, we put
\begin{align}\label{eq:def_ubar_N}
 U_N(z):=U^{(z)}_N(z), \qquad z\in [0,T]\times D,
\end{align}
with $U^{(z)}_N$ defined by \eqref{eq:v.expand}-\eqref{eq:v.0.pide}-\eqref{eq:v.n.pide}.
%\begin{theorem}\label{th:error_estimates_taylor}
%Let Assumptions \ref{assB} and \ref{assC} be in force, and assume also the payoff function $\phi$
%is such that $\phi\equiv\bar{\phi}$ on $D$, with $\bar{\phi}\in
%C_{B,m}^{k,1}\left(\mathbb{R}^d\right)$ for some $k,m\in\N_0$ with $0 \leq m\leq 1$. Then, for any
%$T\in\,]0,\TT_0[$, and for any compact set $K\subset D$, we have
%\begin{align}\label{eq:error_estimate}
% (u-U_N)(t,x;T)  = O \Big( \big( M (T-t) \big)^{\frac{N+k+2}{2}} \Big)  \quad \text{as } t \to T^-,%\qquad 0\leq t<T,\ x\in\mathbb{R}^d,
%\end{align}
%uniformly w.r.t. $x\in K$.
%\end{theorem}
{\begin{theorem}\label{th:error_estimates_taylor} Let Assumptions \ref{assA}, \ref{assB},
\ref{assC} and \ref{assD} be in force. % and assume also that $\phi\in C_{B}^{k}(\R^{d})$ with
%$k\in[0,2r+1]$.
Then for any compact subset $K$ of $D$, we have
\begin{align}\label{eq:error_estimate}
 \left|u(t,x)-U_N(t,x)\right|  \le C (T-t)^{\frac{N+k+1}{2}},\qquad
 (t,x)\in[0,T]\times K,
\end{align}
where $C$ is a positive constant that depends only on
$M,\m,B,T,N,K,\|\psi\|_{C_{B}^{k}(\R^{d})}$ and $\|u\|_{L^{\infty}([0,T]\times D)}$. %If $\phi$
%is merely a continuous function, then estimate \eqref{eq:error_estimate} still holds with $k=0$.
\end{theorem}
}
\noindent Theorem \ref{th:error_estimates_taylor} will be proved in Section
\ref{sec:local_estimate}.
\begin{remark}
As shown in Example \ref{exreg}, for a fixed-strike Asian option we have $\phi \in
C^{3}_{B}(\R^{2})$ and therefore we get $(T-t)^{\frac{N+4}{2}}$ in the error estimate
\eqref{eq:error_estimate}. This is coherent with the previous results proved in
\cite{Gobet2014_weak} in the scalar case for $N\le 2$, %\blu{[Anche ipotesi di regolarita' molto piu' forti]}
and sheds some light on why the order of convergence of Asian call options is improved w.r.t.
their European counterparts, for which the error is of order $(T-t)^{\frac{N+2}{2}}$. When placed
within our framework, this improvement of convergence can be seen as part of a wider phenomenon
related to the intrinsic geometry of Kolmogorov operators.
\end{remark}
\begin{remark}\label{rem:small_vol}
If the coefficients $a_{ij}$, $a_i$ only depend on the first $p_0$ variables, then it is possible to prove the error bounds in \eqref{eq:error_estimate} to be also asymptotic in the limit of small $M$. Precisely,
\begin{align}
 \left|u(t,x)-U_N(t,x)\right|  \le C \big(M(T-t)\big)^{\frac{N+k+1}{2}},\qquad
 (t,x)\in[0,T]\times K,
\end{align}
with $C$ independent of $M$ as $M\to 0^+$. This is the case, for instance, of classical volatility models for Asian options % (see Example \ref{ariasi}),
where the volatility coefficient depends at most on the underlying asset $S_t$ (local volatility)
and on some exogenous factors (stochastic volatility), but not on the average process $A_t$.
\end{remark}
In the global case, when $D=\Rd$, we have some stronger results. Aside from the error bounds in
\eqref{eq:error_estimate} becoming global in space, we are also able to obtain analogous
asymptotic error bounds for the transition density of $X$. We start by observing that when
$D=\Rd$ our assumptions imply that $X$ has a transition density $\G$ that coincides with the
fundamental solution of $\Kc$ as in \eqref{Ac}-\eqref{Kc} (see, for instance, \cite{Polidoro1994}). %In particular, $\G(\cdot;\zeta) \in
%C^{N+2,1}_B ([0,T[\times\Rd)$ for any $\z\in]0,T]\times \Rd$ (see \cite{amrx}).}
% $\Kc$ in \eqref{Ac}-\eqref{Kc} has a fundamental solution $\Gamma$ in $C^{N+2,1}_B$ (see \cite{amrx}).
%\begin{equation}\label{eq:cauchy_problem_fund_sol}
%  \begin{cases}
%    \Kc \Gamma(\cdot,\cdot;T,y)=0,\qquad &\text{on } [0,T[\,\times \Rd, \\
%    \Gamma(T,\cdot;T,y)=\delta_y,\qquad &\text{on } \Rd.
%  \end{cases}
%\end{equation}
We denote by $\barG_{N}$ the {$N$-th order approximation of $\G$} defined as%.  Analogously, for the fundamental solution $\Gamma$ of $(\p_{t}+\Ac)$, we set
\begin{align}
 \barG_{N}(t,x;T,y)=\sum_{n=0}^N {u}_{n}(t,x;T,y) \qquad 0\leq t<T,\ x,y\in\Rd , \label{eq:new_b}
\end{align}
where ${u}_{0}(t,x;T,y) = \Gamma_{0}^{(t,x)}(t,x;T,y)$ in \eqref{Gamma0},
and the correcting terms %the functions
${u}_{n}(t,x;T,y)$ are defined recursively by \eqref{eq:v.n.pide} with $\zb=(t,x)$. % with $\phi=\delta_y$.
We have the following
\begin{theorem}\label{th:error_estimates_taylor_global} Let Assumptions \ref{assA},
\ref{assB}, \ref{assC} and \ref{assD} be in force with $D=\R^{d}$. % and assume also the payoff
%function $\phi\in C_{B}^{k}\left(\mathbb{R}^d\right)$ for some $k\in[0,2r+1]$.
Then, we have
\begin{align}\label{eq:error_estimate_th}
 \left|u(t,x)-U_N(t,x)\right| \le C (T-t)^{\frac{N+k+1}{2}}, \qquad
 (t,x)\in[0,T]\times \R^{d},
\end{align}
where $C$ depends only on $M,\m,B, T,N$ and $\|\phi\|_{C_{B}^{k}(\R^{d})}$. Moreover, for any
$c>1$, we have
\begin{align}\label{th:error_estim_fund_solution}
 \left|\Gamma(t,x;T,y)-\barG_{N}(t,x;T,y)\right|  \leq  C(T-t)^{\frac{N+1}{2}}\Gamma^{cM}(t,x;T,y),\qquad (t,x)\in[0,T[\times
 \R^{d},
\end{align}
where, for any $\Lambda>0$, $\Gamma^{\Lambda}$ denotes the fundamental solution of the constant-coefficient Kolmogorov operator $\Kc^{\Lambda}$ as defined in \eqref{eq:heatoperator},
%\begin{align}\label{eq:heatoperator}
% {\blue \Kc^{\Lambda}=}\frac{\Lambda}{2} \sum_{i=1}^{p_0} \partial^2_{x_i}+Y,
%\end{align}
and $C$ is a positive constant that depends only on $M,\m,B, T,N$ and $c$.
\end{theorem}
\subsection{Proof of Theorem \ref{th:error_estimates_taylor_global}}\label{smalla}%\setcounter{equation}{0}
%Subsection \ref{prel} we provide some preliminary estimates that will be used in the proof in
%Subsection \ref{prel1}.
%\blu{usare la $\phi$ al posto della $\psi$?}
%All of the
%constants appearing in the estimates proved in this section depend on
%Throughout this section $N$, $M_{0}$, $\m$ and $\Tmm$ are fixed. %: thus, for brevity, we omit to
%indicate explicitly such dependence.
%\item[2)]In the special case of constant coefficients, for any $\varepsilon>0$, $m\in \mathbb{N}_0$ and $\ggg\in\mathbb{N}_{0}^{d}$
%\begin{equation}\label{Gaua_zero}
% \left|D_{0}\left(\frac{1}{\sqrt{T-t}}\right)(y-e^{(T-t)B}x)\right|^{m}_B|\, D_x^{\ggg}\Gamma_{0}(t,x;T,y)|\le C \cdot (T-t)^{\frac{m-|\ggg|_B}{2}}\Gamma^{M+\varepsilon}(t,x;T,y),
%\end{equation}
%if $ 0\leq t<T \leq\Tmm,\quad x,y\in\R^{d},$
%\end{enumerate}
%\noindent We are now ready to prove Theorem \ref{th:error_estimates_taylor}.
The proof of Theorem \ref{th:error_estimates_taylor_global} is based on the following two
propositions. The first one provides some Gaussian estimates for the fundamental solution
$\Gamma=\Gamma(t,x;T,y)$ of the operator $\Kc$ in \eqref{Kc}-\eqref{Ac}: for the proof see
\cite{Polidoro1994} and \cite{amrx}. Throughout this section we suppose the assumptions of Theorem
\ref{th:error_estimates_taylor_global} to be in force.
\begin{proposition}\label{lem:gaussian_estimates}
%\begin{enumerate}
%\item[1)]
For any $k\in \R_{\ge 0}$, $c>1$ and $\ggg\in\mathbb{N}_{0}^{d}$, with $|\ggg|_B\le N+ 2$, we have
\begin{equation}\label{Gaua}
 \big[y-e^{(T-t)B}x\big]^{k}_B \big| D_x^{\ggg}\Gamma(t,x;T,y)\big|\le C (T-t)^{\frac{k-|\ggg|_B}{2}}\Gamma^{cM}(t,x;T,y), \qquad 0\leq t<T,\ x,y\in\R^{d},
\end{equation}
where $\Gamma^{cM}$ is the fundamental solution of the operator in \eqref{eq:heatoperator} and $C$
is a positive constant, only dependent on $M,\m,B,T,N,k$ and $c$.
\end{proposition}
}
The following result is proved in Appendix \ref{prel}.
\begin{proposition}\label{prop:der_un}
Let $\phi\in C^{k}_{B}(\Rd)$ with $k\in [0,2r+1]$ and $n\in \N$ with $n\le N$. Then we have
\begin{equation}
 \big|D_x^{\b}u^{(\zb)}_n(t,x)\big|\leq C\:(T-t)^{\frac{k-|\b|_B}{2}}\left((T-t)^{\frac{n}{2}}+\big[x-e^{(t-\bar{t})B}\xb\big]^{n}_B\right),\qquad 0\le t<T,\quad x\in\Rd,
\end{equation}
where $C$ is a constant that depends only on $M,\m,B,T,N,|\b|_B$ and $\|\phi\|_{C^{k}_B(\Rd)}$.
\end{proposition}
\begin{proof}[Proof of Theorem \ref{th:error_estimates_taylor_global}]
To keep formulas at a reasonable size we suppose that the functions $a_i$, $i=1,\dots,p_0,$ in
\eqref{Ac} are identically zero. We first remark that a straightforward computation (see Lemma 6.3
in \cite{LPP4}) shows that
\begin{align}\label{eq:m13_1}
 u(t,x)-U_{N}(t,x) =\sum_{n=0}^N E_n^{(\zb)}(t,x)\Big|_{\bar{z}=(t,x)}.
\end{align}
where
\begin{equation}\label{eq:m13_2}
\begin{split}
 E_n^{(\zb)}(t,x):=&\ \int_t^T \int_{\mathbb{R}^d} \Gamma(t,x;s,\xi) \left(\Kc - {\Kc}^{(\zb)}_n\right) u^{(\zb)}_{N-n}(s,\xi)d \xi d
 s\\
 =&\ \frac{1}{2}\sum_{i,j=1}^{p_0}\int_t^T \int_{\mathbb{R}^d}\Gamma(t,x;s,\xi)
  \Big( a_{ij}(s,\xi) -  \TT_{n}\left(a_{ij},(\zb)\right)(s,\xi))
  \Big) \p_{\xi_i\xi_j}u^{(\zb)}_{N-n}(s,\xi)\, d \xi d s.
\end{split}
\end{equation}
%This result follows from  while the expression for the $E_n$'s by the definition of operators
%$\Kc_i^{(t,x)}$ and $\Kc$.
%{\green (Cosi' $E_n$ non e' definito. Non e' meglio definire direttamente $E_n$ anziche' $E_n^{(\zb)}$?)}
Now, if $k>0$, by Theorem \ref{th:main} and Proposition \ref{prop:der_un} we have
\begin{align}
\!\! \big|E^{(t,x)}_{n}(t,x)\big|&\leq C\int_t^T \!\! \int\limits_{\mathbb{R}^d} \!\Gamma(t,x;s,\xi)\left\|(t,x)^{-1}\circ(s,\x)\right\|^{n+1}_{B}
 (T-s)^{\frac{k-2}{2}}\! \left((T-s)^{\frac{N-n}{2}}+\big[\x-e^{(s-t)B}x\big]^{N-n}_B\right)\! d \xi d s
\intertext{(by Proposition \ref{lem:gaussian_estimates})}
 &\le C  \int_t^T
 (s-t)^{\frac{n+1}{2}}(T-s)^{\frac{k-2}{2}}\left((T-s)^{\frac{N-n}{2}}+(s-t)^{\frac{N-n}{2}}\right) d s \\
 &\le C \: (T-t)^{\frac{N+k+1}{2}}
\end{align}
where we have used the identity
  $$\int_{t}^{T}(T-s)^{n}(s-t)^{k}\, d s=\frac{\Gamma_E (k+1) \Gamma_E (n+1)}{\Gamma_E (k+n+2)}(T-t)^{k+n+1},\qquad n,k>-1,$$
with $\Gamma_E$ denoting the Euler Gamma function.
The case $k=0$ can be handled similarly  performing first an integration by parts in \eqref{eq:m13_2}.

Finally, estimate \eqref{th:error_estim_fund_solution} can be proved by a straightforward
modification of the proof of \eqref{eq:error_estimate_th}, using also the Chapman-Kolmogorov
equation. We omit the details for brevity.
\end{proof}
\begin{remark}\label{rem:greeks}
%Under the additional hypothesis $k>0$, we have
%\begin{equation}
%\big|D_x^{\a}u(t,x)-D_x^{\a}U_{N}(t,x)\big|\leq C(T-t)^{\frac{N+k+1-|\a|_B}{2}}, \qquad |\a|_B\leq 2.
%\end{equation}
%Non dovrebbe essere
Under the assumptions of Theorem \ref{th:error_estimates_taylor_global}, we have also error bounds
for the approximation of the derivatives of $u$; precisely, we have
\begin{equation}\label{eq:greeks}
\big|D_x^{\a}u(t,x)-D_x^{\a}U^{(\bar{z})}_{N}(t,x)|_{\bar{z}=(t,x)}\big|\leq
C(T-t)^{\frac{N+k+1-|\a|_B}{2}}, \qquad |\a|_B\leq N.
\end{equation}
The proof of this formula is analogous to the proof of Theorem
\ref{th:error_estimates_taylor_global}, once $D_x^{\a}$ is applied to the representation formulas
\eqref{eq:m13_1} and \eqref{eq:m13_2}. %(see also \cite{Lorig} for similar results in the uniformly parabolic case).
When $u(t,x)$ represents the price of an arithmetic Asian
option, %, see Examples \ref{ariasi},
formula \eqref{eq:greeks} provides error bounds on the approximate sensitivities %with respect to
%the underlying asset price,
%In particular, in the Example \ref{exreg}, we obtain bounds on the error committed approximating derivatives of the price of an Asian option with respect to underlying's price,
or, as they are usually called in finance, the Greeks. % Delta and Gamma. %{\green ($\G$ e' anche la densita, si potrebbero chiamare Delta e Gamma)}.
For instance, in the case of a fixed-strike Asian option (see Example \ref{exreg}), we have $k=3$
and thus
%Precisely, since in this case we have $k=3$, we get
\begin{equation}
 \big|\text{Delta} - \p_{x_1}U_N^{(\zb)}|_{\bar{z}=(t,x_{1},x_{2})}\big|\leq C(T-t)^{\frac{N+3}{2}},
 \qquad \big|\text{Gamma} - \p_{x_1,x_1}U_N^{(\zb)}|_{\bar{z}=(t,x_{1},x_{2})}\big|\leq C(T-t)^{\frac{N+2}{2}},
\end{equation}
where $\text{Delta}:=\p_{x_1}u$ and $\text{Gamma}:=\p_{x_1,x_1}u$.
\end{remark}

\subsection{Proof of Theorem \ref{th:error_estimates_taylor}}\label{sec:local_estimate} %In this
%section we prove Theorem \ref{th:error_estimates_taylor}, that is, we extend our previous results
%to the the case of a degenerate Kolmogorov Operator not necessarily defined on the whole space.
%Precisely, we show that, even if hypothesis \eqref{cond:parabolicity} holds just on a compact
%domain, an estimate of type \eqref{eq:error_estimate} still holds \emph{in such domain}.
Throughout this section we suppose the assumptions of Theorem \ref{th:error_estimates_taylor} to
be in force. The proof of Theorem \ref{th:error_estimates_taylor} is based on
%The main tool in the proof of Theorem \ref{th:error_estimates_taylor} are
some estimates on short cylinders initially introduced in \cite{Safonov1998} for uniformly parabolic
operators and later generalized to Kolmogorov operators  in \cite{Cinti2009135}.

First, we introduce the ``cylinder'' of radius $R$ and height $h$ centered in $(s,\x)\in \Rdd$ and
its lateral and parabolic boundaries, respectively:
\begin{align}
 H_{h,R}(s,\x)&:= \{(t,x)\in \Rdd \:|\:s-h<t<s,\:[x-e^{(t-s)B}\x]_B<R\},\\
% D_R(s,\x)&:=\{(t,x)\in \Rdd \:|\: t=s,\: |x-\x|_B<R\}\\
% D_R(s,\x,h)&:= \{(t,x)\in \Rdd \:|\: t=s-h, \:|x-e^{-h B}\x|_B<R\}\\
 \Sigma_{h,R} (s,\x)&:= \{(t,x)\in \Rdd \:| \: s-h<t<s,\: [x-e^{(t-s)B}\x]_B=R\},\\
 \p_P H_{h,R} (s,\x)&:= \Sigma_{h,R} (s,\x) \cup \{(s,x)\in \Rdd \:|\: [x-\x]_B<R\}.
\end{align}
We explicitly observe that these cylinders are invariant with respect to the left translations in
$\mathcal{G}_{B}$, {meaning that $z\circ H_{h,R}(\z)=H_{h,R}(z\circ \z)$ for any $z,\z \in \Rdd$}.
We also recall the following inequality (see Proposition 2.1 in \cite{Manfredini}):
\begin{equation}\label{eq:estimate_norm_2}
 \|z\circ \z\|_{B}\le c_{B}\left(\|z\|_{B}+\|\z\|_{B}\right),\qquad z,\z\in\Rdd,
\end{equation}
where $c_{B}\ge 1$ is a constant that depends only on the matrix $B$. In particular, taking
$z=(0,x)$ and $\z=(t,0)$, \eqref{eq:estimate_norm_2} implies that
\begin{equation}\label{eq:estimate_norm_2bis}
 [e^{tB}x]_B\le \|(t,e^{tB}x)\|_{B}=\norm{z\circ\z}_B\le c_B\big(|t|^{\frac{1}{2}}+[x]_B\big),\qquad t\in\R,\ x\in\Rd.
\end{equation}
\begin{lemma}\label{lem:estimate_short_cylinder} There exist $C>0$, $\eps\in\, ]0,1[$, only dependent on
$M,\mu,B$, and a nonnegative function $v\in C([0,T]\times\Rd)\cap C^{2,1}_{B,\text{\rm loc}}$ such
that, for every $R>0$ we have
\begin{align}\label{eq:lem_cylinder1}
 &{\Kc}v(t,x)= 0 & (t,x)\in H_{\varepsilon R^2,R}(T,0),\\  \label{eq:lem_cylinder2}
 &v(t,x)\geq 1, & (t,x)\in \Sigma_{\varepsilon R^2,R}(T,0),\\ \label{eq:lem_cylinder3}
 & v(t,x)\leq  C\exp\left(-\frac{ R^2}{{C(T-t)}}\right) & (t,x)\in H_{\varepsilon R^2,\frac{R}{8 c_B^2}}(T,0),
\end{align}
where $c_{B}$ is the constant in \eqref{eq:estimate_norm_2}.
\end{lemma}
%{\blue {\bf [Proposta cambio notazione]} se ponessimo $\d:=\frac{1}{8 c_B^2}$? avere un pedice come $=\frac{R}{8 c_B^2}$  nel cilindro è brutto e poco visibile, meglio $\d R$ secondo me...}
\begin{proof}
%The result is based on the following Gaussian estimates for the fundamental solution ${\G}$ of
%${\Kc}$ proved in \cite{Polidoroba}: there exist two positive constants $c^{-}$ and $c^{+}$, only
%{depending on $M,\mu$ and $B$,} such that
%\begin{equation}\label{eq:estimate_gammatilde}
% c^{-}\G^{-}(t,x;s,\x)\leq {\G}(t,x;s,\x) \leq c^{+}\G^{+}(t,x;s,\x), \qquad 0\leq t<s\leq T,\ x,\x\in\Rd.
%\end{equation}
%Here $\G^{-}$ and $\G^{+}$ denote the fundamental solutions of the constant coefficients
%Kolmogorov operators $\Kc^{\frac{\m M}{2}}$ and $\Kc^{2M}$ defined as in \eqref{eq:heatoperator},
%respectively.
{Let $\G$ denote the fundamental solution of ${\Kc}$ in \eqref{Kc}: $\G$ can be thought as the
transition density of a dummy process $\tilde{X}$ whose infinitesimal generator is ${\Ac}$ and can
be used to approximate the original process $X$ locally on $D$. The proof of the lemma is based on
a Gaussian upper bound for $\G$. More precisely, since ${\Kc}$ is a
global Kolmogorov operator, by Proposition \ref{lem:gaussian_estimates} we have:%We introduce a
%dummy process $\tilde{X}$ whose infinitesimal generator is operator ${\Kc}$ in \eqref{Kc}. The
%proof of the lemma is based on a Gaussian upper bound for the transition density $\G$ of
%$\tilde{X}$. More precisely, since ${\Kc}$ is a global Kolmogorov operator, by Proposition
%\ref{lem:gaussian_estimates} we have:
} there exists a positive constant $c^{+}$, only {depending
on $M,\mu$ and $B$,} such that
\begin{equation}\label{eq:estimate_gammatilde}
 {\G}(t,x;s,\x) \leq c^{+}\G^{\Lambda}(t,x;s,\x), \qquad 0\leq t<s\leq T,\ x,\x\in\Rd,
\end{equation}
where $\G^{\Lambda}$ is the fundamental solution of the constant coefficients Kolmogorov operator
in \eqref{eq:heatoperator} and $\Lambda$ is strictly greater {than} $M$, say $\Lambda=2M$.

Next, we set
\begin{equation}
 v(t,x)=2\int\limits_{\Rd}{\G}(t,x;T,y)\chi_{R}(y) dy,\qquad t<T,\ x\in\R^{d},
\end{equation}
where $\chi_{R%,\d_1
}\in C^{\infty}(\Rd,[0,1])$ is a cut-off function such that $\chi_{R%,\d_1
}(y)=0$ if $[y]_B< \frac{R}{2} %\d_1 R c_B^2
$
and $\chi_{R%,\d_1
}(y)=1$ if $[y]_B> \frac{3}{4}R$. By definition, it is clear that $v$ satisfies
\eqref{eq:lem_cylinder1}. Moreover, we have
\begin{align}\label{limi}
 \lim_{t\to T^{-}}v(t,x)=2\chi_{R}(x)=2,
\end{align}
{\it uniformly w.r.t. $x\in\R^{d}$ such that $[x]_{B}=R$:} this follows by noting that
\begin{align}
 \left|v(t,x)-2\chi_{R}(x)\right|&\le 2\int\limits_{\Rd}{\G}(t,x;T,y)\left|\chi_{R}(y)-\chi_{R}(x)\right| dy && \\
%\intertext{(by \eqref{eq:estimate_gammatilde})}
 &\le 2c^{+}\int\limits_{\Rd}{\G}^{\Lambda}(t,x;T,y)\left|\chi_{R}(y)-\chi_{R}(x)\right| dy. && \text{(by \eqref{eq:estimate_gammatilde})}
\end{align}
Now, by \eqref{limi} there exists $\varepsilon>0$, which we
can safely assume to be less than { $\frac{1}{16 c_B^4}$ and $\frac{1}{64 c_B^2}$}, %$\frac{\d^2}{4}$, % and $\big(\frac{\d c_B R}{4}\big)^2$,
such that \eqref{eq:lem_cylinder2} holds.

The proof of \eqref{eq:lem_cylinder3} depends on the reverse triangle inequality for the norm
$[\cdot]_B$:
\begin{equation}
\label{eq:estimate_norm_increment}
 [y-e^{tB}x]_B \geq \frac{1}{c_B}[y]_B - c_B\big(|t|^{\frac{1}{2}}+[x]_B\big),\qquad t\in\R,\ x,y\in\Rd,
\end{equation}
whose proof is an easy consequence of  \eqref{eq:estimate_norm_2bis}. In particular, if { $[y]_B\geq \frac{R}{2} $} %$[y]_B\geq \d_1 R {c^2_B}$
and { $(t,x)\in H_{{ \eps R^2,\frac{R}{8 c_B^2} }}(T,0)$}, %$(t,x)\in H_{\d R}(T,0;\eps R^2)$
then in light of the first bound for $\eps$ we get
\begin{equation}\label{st}
 [y-e^{(T-t)B}x]_B \geq { \frac{R}{8 c_B}}. %\frac{1}{2}\d c_B R.
\end{equation}
Hence, for such $(t,x)$
%\in H_{\frac{\d R}{2}}(0,T,\varepsilon R^2)$
we get
\begin{align}
 v(t,x)&\leq 2c^{+}\int\limits_{\Rd}\G^{\Lambda}(t,x;T,y){\chi_{R}}(y) dy \leq
 2c^{+}\int\limits_{{ [y]_B\geq \frac{R}{2} }%[y]_B\geq\d c_B^{2} R
 }\G^{\Lambda}(t,x;T,y)dy\\
    &=\frac{2c^{+}(2\pi)^{-\frac{d}{2}}}{\sqrt{|\mathbf{C}(T-t)|}}\int\limits_{{ [y]_B\geq \frac{R}{2} }%[y]_B\geq\d c_B^{2} R
    } \exp\left(-\frac{1}{2}\langle \mathbf{C}^{-1}(T-t)(y-e^{(T-t)B}x),(y-e^{(T-t)B}x)\rangle\right) dy
\intertext{(by \eqref{st} and denoting by $\mathbf{C}$ the matrix in \eqref{eq:covariance_mean}
with $A_{0}=\Lambda I_{p_{0}}$ and $I_{p_{0}}$ being the $(p_{0}\times p_{0})$ identity  matrix)}
    &\leq \frac{2c^{+}(2\pi)^{-\frac{d}{2}}}{\sqrt{|\mathbf{C}(T-t)|}}\int\limits_{ [y-e^{(T-t)B}x]_B\geq {\frac{R}{8 c_B}}%[y-e^{(T-t)B}x]_B\geq\frac{\d {c_B} R}{2}
    }\exp\left(-\frac{1}{2}\langle
    \mathbf{C}^{-1}(T-t)(y-e^{(T-t)B}x),(y-e^{(T-t)B}x)\rangle\right) dy
\intertext{(by the change of variables $\eta= D_0(\frac{1}{\sqrt{T-t}})(y-e^{(T-t)B}x)$ and the
homogeneity relation \eqref{eq:matrix_homogeneity})}
    &=\frac{2c^{+}(2\pi)^{-\frac{d}{2}}}{\sqrt{|\mathbf{C}(1)|}}\int\limits_{[\eta]_B\geq {\frac{R}{8 c_B \sqrt{T-t}}} %\frac{\d c_B R }{2\sqrt{T-t}}
    }\exp\left(-\frac{1}{2}\langle \mathbf{C}^{-1}(1)\eta,\eta\rangle\right)
    d\eta. \label{eq:ste400}
\end{align}
%In the last line, we exploited  and
%that $[D_0(\lambda)x]_B=\lambda [x]_B$ for every $x\in \Rd$ and for all $\lambda>0$, together with
%the change of variable . $A$ is the region $\{\eta\in \Rd\,:\,[\eta]_B\geq\frac{\d c_B R }{2\sqrt{T-t}} \}$.
%\begin{equation}
%v(t,x)\leq \frac{2c^{+}(2\pi)^{-\frac{d}{2}}}{c^{-}\sqrt{|C(1)|}}\int\limits_{A}\exp\left(-\frac{1}{2}\langle C^{-1}(1)\eta,\eta\rangle\right) d\eta.
%\end{equation}
Since we are assuming $T-t \leq \eps R^2$, thanks to the %previous choice of
second bound on $\varepsilon$ we have $[\eta]_B\geq {\frac{R}{8 c_B \sqrt{T-t}}} %\frac{\d Rc_{B}}{2\sqrt{T-t}}
\geq 1$ and thus, there exists $C_0>0$ only dependent on $\mu,M,B$, such that
\begin{align}
\langle \mathbf{C}^{-1}(1)\eta,\eta\rangle \geq C_0|\eta|^2 & = { C_0 \sum_{j=1}^d \frac{|\eta_j|^2}{[\eta]^{2\sigma_j}_B} [\eta]^{2\sigma_j}_B = C_0 \sum_{j=1}^d \bigg(\frac{|\eta_j|^{{1}/{\s_j}}}{[\eta]_B} \bigg)^{2 \s_j} [\eta]^{2\sigma_j}_B }\\
& { \geq C_0 [\eta]^{2}_B \sum_{j=1}^d \bigg(\frac{|\eta_j|^{{1}/{\s_j}}}{[\eta]_B} \bigg)^{2 (2r+1)} \geq \frac{C_0}{d^{4r+1}} [\eta]^{2}_B \bigg( \sum_{j=1}^d \frac{|\eta_j|^{{1}/{\s_j}}}{[\eta]_B} \bigg)^{2 (2r+1)}    = \frac{C_0}{d^{4r+1}} [\eta]^{2}_B. }
\end{align}
{ Setting $C_1:= \frac{C_0}{d^{4r+1}}$ we get}
%from which we can deduce
\begin{align}
 \int\limits_{[\eta]_B\geq { \frac{R}{8 c_B \sqrt{T-t}}} %\frac{\d c_B R }{2\sqrt{T-t}}
 }\!\! \exp\Big(-\frac{1}{2}\langle \mathbf{C}^{-1}(1)\eta,\eta\rangle\Big) d\eta
 &\leq \int\limits_{[\eta]_B\geq {\frac{R}{8 c_B \sqrt{T-t}}}%\frac{\d c_B R }{2\sqrt{T-t}}
 }\exp\Big(-\frac{1}{2}C_1[\eta]^2_B \Big)d\eta\\
  &\leq \max_{[y]_B\geq {\frac{R}{8 c_B \sqrt{T-t}}}%\frac{\d c_B R }{2\sqrt{T-t}}
  }
  \exp\Big(-\frac{1}{4}C_1[y]^2_B\Big)
  \int\limits_{[\eta]_B\geq {\frac{R}{8 c_B \sqrt{T-t}}}%\frac{\d c_B R }{2\sqrt{T-t}}
  }\!\exp\Big(-\frac{1}{4}C_1[\eta]^2_B\Big) d\eta\\
  &{ \leq }\exp\Big(-\frac{ C_1 R^2}{2^8 c^2_B(T-t)}\ \Big) \int_{\Rd%\int\limits_{\Rd[\eta]_B\geq {\frac{R}{8 c_B \sqrt{T-t}}}%\frac{\d c_B R }{2\sqrt{T-t}}
  }\exp\Big(-\frac{1}{4}C_1[\eta]^2_B\Big) d\eta,
\end{align}
which, combined with \eqref{eq:ste400}, proves \eqref{eq:lem_cylinder3}.% follows.
 \end{proof}
% {\blue
%Before to go ahead with the proof of Theorem \ref{th:error_estimates_taylor}, we provide the following Feynman-Kac representation formula for Kolmogorov operators.
%\begin{lemma}\label{lem:feynman-kac} Let $T,h,R>0$ and $\xi\in\Rd$, and $w\in C\big(\,\overline{H_{h,R}(T,\xi)}\,\big)\cap C^{2,1}_{B,\text{\rm loc}}\big(H_{h,R}(T,\xi)\big)$ such
%that%, for every $R>0$ we have
%\begin{equation}\label{cauchy-dirich_b}
% \Kc w=0,\qquad \text{on } H_{h,R}(T,\xi).
%\end{equation}
%Then, it holds:
%\begin{equation}
%w(t,x) = \Eb\big[w\big(\tau,X^{t,x}_{\tau}\big)\big], \qquad (t,x)\in H_{h,R}(T,\xi).
%\end{equation}
%Here above, $(X^{t,x}_{s})_{t\leq s\leq T}$ is a solution, on some probability space $(\Omega,\Fc,\Pb)$,  to the SDE
%\begin{equation}
%  \begin{cases}
%\dd X_s =( B X_s + \mu(s,X_s)) \dd s +  \sigma(s,X_s) \dd W_s \qquad &\text{on } ]t,T], \\
%    X_t = x,
%  \end{cases}
%\end{equation}
%with $\mu:[t,T]\times\Rd \to \Rb^{p_0}$ and $\sigma:[t,T]\times\Rd \to \Rd\times \Rb^{p_0}$ such that
%%\begin{equation}\label{cauchy-dirich}
%%  \begin{cases}
%%    \Kc w=0,\qquad &\text{on } H_{h,R}(T,\xi), \\
%%    w=g,\qquad &\text{on } \p_P H_{h,R} (T,\xi),
%%  \end{cases}
%% \end{equation}
% \begin{equation}
% \mu_i= a_i, \qquad \big(\sigma^\top\sigma\big)_{ij} = a_{ij},  \qquad  i,j=1,\cdots,p_0,
% \end{equation}
% and $\tau$ is the stopping-time
% \begin{equation}
% \tau:= \min\{ s>t : \left(s,X^{t,x}_{s}\right) \in  \partial_P H_{h,R}(T,\xi) \}.
% \end{equation}
%\end{lemma}
% \begin{proof}
% It's an application of Ito's formula
% \end{proof}
% }
\begin{proof}[Proof of Theorem \ref{th:error_estimates_taylor}]
Since the statement is a short-time estimate on a compact subset, it is enough to prove
\eqref{eq:error_estimate} for $(t,x)\in H_{\eps R^2,R}(T,\xi)\subseteq {]0,T[\times D}$ for
suitably small $\eps,R>0$.
%First of all it is sufficient to show that the claim for the special case of a cylindrical (?) domain $H_{\eps R^2,R}(T,\xi)\subset D$ as any compact set can be covered by a finite number of them.
Secondly, we can
suppose $\x=0$. In fact, if $u$ is a solution to $\Kc u=0$ in $H_{\eps R^2,R}(T,\xi)$ then $w(t,x)=u(t,
x-e^{-T B}\x)$ solves %$\tilde{\Kc} w=0$ in
on $H_{\eps R^2,R}(T,0)$ the operator obtained through $\Kc$ by translating its coefficients.
%where , of course, the coefficients of $\Kc$ are
%computed in $(t, x-e^{-T B}\x)$. Moreover, we have to show that the bound in formula \eqref{th:error_estimates_taylor} holds just for small times, say $T-t<\varepsilon R^2$, with $\varepsilon$ as in the proof of the previous Lemma.
%We denote by $u$ the classical bounded solution to the Cauchy problem on $D$ \eqref{equaz1} with terminal condition $\phi$ while by

Let us denote by ${u}^{\psi}$ the unique solution (with polynomial growth) to the Cauchy problem
\begin{equation}\label{equaz2}
  \begin{cases}
    \Kc f=0,\qquad &\text{on } [0,T[\times \Rd, \\
    f(T,\cdot)=\psi,\qquad &\text{on } \Rd,
  \end{cases}
\end{equation}
with $\psi$ as in Assumption \ref{assD}, and by $U_N^{\psi}$ its $N$-th order approximation as defined in Section \ref{sec:approximating1}. By triangular inequality we have
\begin{equation}\label{eq:ste405}
|u-U_N|\leq
|u-{u}^{\psi}|+|{u}^{\psi}-U_N^{\psi}|+|U_N^{\psi}-U_N|.
\end{equation}
We now aim at estimating each of the terms in the sum above.

We start with $|u-{u}^{\psi}|$.  %Recall that $\phi\equiv {\psi}$ on $D$.
Let $v$ be the function appearing in Lemma \ref{lem:estimate_short_cylinder}. {By Proposition
\ref{prop:solution_cauchy_pb} and \eqref{eq:lem_cylinder1},}  $u-{u}^{\psi}$ and $v$ solve
%\begin{equation}%\label{}
%  \begin{cases}
%    \Kc w=0,\qquad &\text{on } [0,T[\times D, \\
%    w(T,\cdot)=\varphi,\qquad &\text{on } D.
%  \end{cases}
%\end{equation}
%
$\Kc w=0$ in $ H_{\eps R^2,R}(T,0)$ and are continuous on $\overline{H_{\eps R^2,R}(T,0)}$. Moreover, $(u-{u}^{\psi})(T,x)=0$ if
$[x]_B<R$, {and thus, by setting %if we set
 \begin{equation}
 C_1:=\max\limits_{\Sigma_{\eps R^2, R}(T,0)}|u-{u}^{\psi}|,
 \end{equation}
we get $|u-{u}^{\psi}|\leq C_1 v $ on $\p_P H_{\eps R^2,R}(T,0)$. Therefore, by the Feynman-Kac
theorem we have
\begin{align}
 \big| \big(u-u^{\psi}\big)(t,x) \big| = \big| \Eb_{t,x} \big[ (u-u^{\psi})\big(\tau,X_{\tau}\big) \big] \big| \leq C_1 \Eb_{t,x} \big[ v\big(\tau,X_{\tau}\big) \big] = C_1 v(t,x),
\end{align}
where $\tau$ denotes the exit time from $H_{\eps R^2,R}(T,0)$ of the process $(s,X_{s})$ starting
from $(t,x)\in H_{\eps R^2,R}(T,0)$.
%where we set
%by setting %if we set
 %\begin{equation}
 %C_1:=\max\limits_{\Sigma_{\eps R^2, R}(T,0)}|u-{u}^{\psi}|,
 %\end{equation}
%we get $|u-{u}^{\psi}|\leq C_1 v $ on $\p_P H_{\eps R^2,R}(T,0)$ and, by the maximum principle combined with
By estimate \eqref{eq:lem_cylinder3} of Lemma \ref{lem:estimate_short_cylinder} we obtain}
\begin{equation}\label{eq:ste402}
\big|(u-{u}^{\psi})(t,x)\big|\leq C_1 C_2 \exp\left(-\frac{ R^2}{{C_2(T-t)}}\right), \qquad    (t,x)\in H_{\eps R^2,\frac{R}{8 c_B^2}}(T,0),
\end{equation}
with $C_2>0$ depending only on $M,\mu,B$.

We continue by estimating $|{u}^{\psi}-U_N^{\psi}|$. By Theorem \ref{th:error_estimates_taylor_global} there exists $C_3>0$, only dependent on $M,\m,B, T,N$ and $\|\psi\|_{C_{B}^{k}(\R^{d})}$, such that
\begin{align}\label{eq:ste406}
 \big|u^{\psi}(t,x)-U^{\psi}_N(t,x)\big| \le C_3 (T-t)^{\frac{N+k+1}{2}}, \qquad
 (t,x)\in[0,T]\times \R^{d}.
\end{align}

We conclude by estimating $|U_N^{\psi}-U_N|$. First observe that, by \eqref{e10}, for any multi-index $\alpha\in \N^d_0$ we have
\begin{equation}
 D^{\alpha}_x (u^{(\bar{z})}_0-u^{(\bar{z}),\psi}_0\big)  (t,x) = D^{\alpha}_x \int_{\Rd} \G^{(\bar{z})}_0(t,x;T,y) \big(\phi(y) - \psi(y)\big) dy  = \int_{\Rd} D^{\alpha}_x \G^{(\bar{z})}_0(t,x;T,y)\big(\phi(y) - \psi(y)\big) dy,
\end{equation}
with $\G^{(\bar{z})}_0$ as in \eqref{Gamma0}. Now, $\G^{(\bar{z})}_0$ is the fundamental solution of the constant-coefficients Kolmogorov operator $\Kc^{(\bar{z})}_0$ in \eqref{K0}, for which Assumptions \ref{assA}, \ref{assB} and \ref{assC} are trivially satisfied. Therefore, the bounds in Lemma \ref{lem:gaussian_estimates} also apply to $\G^{(\bar{z})}_0$ and yield
\begin{equation}\label{eq:ste401}
\big|D^{\alpha}_x (u^{(\bar{z}),\psi}_0 - u^{(\bar{z})}_0   \big)  (t,x) \big| \leq C_4 (T-t)^{-\frac{|\alpha|_B}{2}} w(t,x), \qquad \bar{z}\in\Rdd, \ (t,x)\in [0,T[\times \Rd,
\end{equation}
with
\begin{equation}
w(t,x):=  \int_{\Rd}  \G^{2M}(t,x;T,y) \big|\big(\phi(y) - \psi(y)\big)\big| dy,
\end{equation}
where $\G^{2M}$ is the fundamental solution of the Kolmogorov operator $\Kc^{2M}$ as in \eqref{eq:heatoperator}, and $C_4>0$ only depends on $M,\m,B,T,|\alpha|_B$.
Now note that, by \eqref{eq:def_ubar_N} and \eqref{eq:un}, we have
\begin{equation}
\big( U_N^{\psi}-U_N\big) (t,x)= \big(u^{(\bar{z}),\psi}_0 - u^{(\bar{z})}_0   \big)  (t,x) + \sum_{n=1}^{N} \Lc^{(\bar{z})}_n \big( u_0^{(\bar{z}),\psi} - u^{(\bar{z})}_0  \big)(t,x) \bigg|_{\bar{z}=(t,x)}.
\end{equation}
Thus by Lemma \ref{lem:Ln} with \eqref{eq:ste401} we get
\begin{equation}
\big|\big( U_N^{\psi}-U_N\big) (t,x)\big| \leq C_5 |w(t,x)|,\qquad  (t,x)\in [0,T[\times \Rd,
\end{equation}
where $C_5>0$ only depends on $M,\m,B,T$ and $N$. By repeating step by step the same proof of \eqref{eq:ste402} it is straightforward to obtain an estimate for $|w(t,x)|$ analogous to \eqref{eq:ste402}, which finally yields
\begin{equation}\label{eq:ste403}
\big|\big( U_N^{\psi}-U_N\big) (t,x)\big|\leq C_5 C_6 C_7 \exp\left(-\frac{ R^2}{{C_7(T-t)}}\right), \qquad    (t,x)\in H_{\eps R^2,\frac{R}{8 c_B^2}}(T,0),
\end{equation}
with $C_7>0$ depending only on $M,\m,B,T,N$, and
\begin{equation}
 %$
 C_6:=\max\limits_{\Sigma_{\eps R^2, R}(T,0)}\big|w \big|.
%$
 \end{equation}
Plugging \eqref{eq:ste402}-\eqref{eq:ste406}-\eqref{eq:ste403} into \eqref{eq:ste405} yields
\eqref{eq:error_estimate} for $(t,x)\in H_{\eps R^2,\frac{R}{8 c_B^2}}(T,0)$ and concludes the
proof.
%We recall that $U_N$ stands for the $N-$th order approximant of the true solution $u$ and that superscripts indicate the final condition. Then a procedure totally analouguos to the one above can be applyed to bound $U_N^{\psi}-U_N^{\phi}$. In fact, similarly to $u-{u}^{\psi}$, this function solves $\Kc_0 w=0$ in $D$ and $(U^{\psi}_N-u_N)(T,x)=0$ if $[x]_B<R$,the only difference being  operator $\Kc_0$ instead of $\Kc$. Since $\Kc_0$ satisfies the same hyphothesis as $\Kc$, the bound holds true with a new constant $C_0$ and function $v_0$.

%Now, if we specialize to the cylinder $H_{\frac{\d R}{2}}(T,0;\varepsilon R^2)$ given by Lemma \ref{lem:estimate_short_cylinder}, we can estimate the error between the true solution of problem \eqref{equaz1}, $u$, and the $N$-th order approximation of the true solution to the global problem with  $U_N^{\psi}$, as follow
%\begin{equation}
%|u-U_N^{\phi}|\leq |u-{u}^{\psi}|+|{u}^{\psi}-U_N^{\psi}|+|U_N^{\psi}-U_N^{\phi}|\leq C_1 v_1 + C(T-t)^{\frac{N+k+1}{2}} +C_2 v_2,
%\end{equation}
%and \eqref{th:error_estimates_taylor} follows by \eqref{eq:lem_cylinder3} and Theorem \ref{th:error_estimates_taylor}.
\end{proof}
%%%%%%%%%%%%%%%%%%%%%%%%%%%%%%%%%%%%%%%%%%%%%%%%%%%%%%%%%
%
%           APPENDIX
%
%%%%%%%%%%%%%%%%%%%%%%%%%%%%%%%%%%%%%%%%%%%%%%%%%%%%%%%%%
\appendix
%\counterwithin{theorem}{section}

\section{Analytical approximation formulas}\label{appendix:proof}%\setcounter{equation}{0}
\numberwithin{theorem}{section}
%\blu{\bf [Cambiare numerazione teoremi]}

We show that the functions $u^{(\bar{z})}_n$ in \eqref{eq:v.0.pide}-\eqref{eq:v.n.pide} can be
explicitly computed at any order. It is clear that the leading term $u^{(\bar{z})}_0$ is given by
\begin{align}\label{e10}
 u^{(\bar{z})}_0(t,x) &=  \int_{\mathbb{R}^d} \Gamma^{(\bar{z})}_{0}(t,x;T,y) \phi(y)d y,\qquad (t,x) \in [0,T[\times \Rd , %t<T,\ x\in\mathbb{R}^d.
\end{align}
where $\Gamma^{(\bar{z})}_{0}$ is the Gaussian density in \eqref{Gamma0}. For $n\in\N$ with $n\leq N$, the
explicit representation for the correcting terms $u^{(\bar{z})}_n$ can be derived using the
following notable symmetry properties of $\Gamma^{(\bar{z})}_{0}$.
\begin{lemma}\label{lem:symmetries}
For any $x,y\in\mathbb{R}^d$, $ t<s$ and $\bar{z}=(\bar{t},\bar{x})\in \Rdd$, we have
\begin{align}\label{prop:deriv_bis}
 \nabla_x \Gamma^{(\bar{z})}_{0}(t,x;s,y)
    &= -e^{(s-t)B^*}\nabla_{y} \Gamma^{(\bar{z})}_{0}(t,x;s,y), \\
\label{prop:polyn1_bis}
 y\, \Gamma^{(\bar{z})}_{0}(t,x;s,y) & =  \Mc^{(\bar{z})}(s-t,x)\Gamma^{(\bar{z})}_{0}(t,x;s,y),
\end{align}
where $\Mc^{(\bar{z})}(t,x)$ is the operator defined as
\begin{equation} \label{eq:M}
 \Mc^{(\bar{z})}(t,x) = e^{tB}\left(x + \Mv_{\bar{z}}(t)\nabla_{x}\right),\qquad
 \Mv_{\bar{z}}(t)=   e^{-tB}\Cv_{\bar{z}}(t)e^{-tB^*}.
\end{equation}
\end{lemma}
\begin{proof}
Using the explicit expression of $\Gamma^{(\bar{z})}_{0}$, the proof is a direct computation.
\end{proof}
%{\blue
%\begin{remark}
%As it will be clarified later, the symmetries \eqref{prop:deriv_bis}-\eqref{prop:polyn1_bis} are
%``compatible'' with the increments of the Taylor expansion \eqref{eq:tay_int}. {\bf [Dove?]} This fact confirms
%the natural role played by the intrinsic functional spaces $C^{N,\a}_{B}$ in our analysis.
%\end{remark}}
The following result provides an explicit representation of $u^{(\bar{z})}_n$ in
\eqref{eq:v.expand}: remarkably, it can be written as a finite sum of spatial derivatives acting
on $u^{(\bar{z})}_0$.
\begin{theorem}\label{th:un_general_repres}
Let Assumptions \ref{assA}, \ref{assB} and \ref{assC} be in force. Then, for any $n\in\N$ with $
n\leq N$, and for any $\bar{z}\in\Rdd$,  we have
\begin{equation}\label{eq:un}
 u^{(\bar{z})}_n(t,x)=  \Lc^{(\bar{z})}_n(t,T,x) u^{(\bar{z})}_0(t,x),\qquad (t,x)\in  [0,T[\times
 \Rd.
\end{equation}
In \eqref{eq:un}, $\Lc^{(\bar{z})}_n(t,T,x)$ denotes the differential operator
\begin{align}\label{eq:def_Ln}
  \Lc^{(\bar{z})}_n(t,T,x)=  \sum_{h=1}^n \int_{t}^T d s_1 \int_{s_1}^T d s_2 \cdots \int_{s_{h-1}}^T d s_h
      \sum_{i\in I_{n,h}}\Gc^{(\bar{z})}_{i_{1}}(t,s_1,x) \cdots \Gc^{(\bar{z})}_{i_{h}}(t,s_h,x) ,
\end{align}
where
\begin{align}\label{eq:def_Ln_bis}
 I_{n,h}=  \{i=(i_{1},\dots,i_{h})\in\mathbb{N}^{h} \mid i_{1}+\dots+i_{h}=n\},\qquad 1\le h \le n,
\end{align}
and %the operator $\Gc^{(\bar{z})}_{n}(t,s,x)$ is defined as
\begin{align}
 \Gc^{(\bar{z})}_{n}(t,s,x)  =&\ \frac{1}{2}\sum_{i,j=1}^{p_0}
 \big(\TT_n  (a_{ij},\bar{z}) - \TT_{n-1}(a_{ij},\bar{z})\big) \big(s,\Mc^{(\bar{z})}(s-t,x)\big) \big(e^{-(s-t)B^*}\nabla_x\big)_i \big(e^{-(s-t)B^*}\nabla_x\big)_j\\
  &  + \sum_{i=1}^{p_0}  \big(\TT_{n-1} (a_{i},\bar{z}) - \TT_{n-2} (a_{i},\bar{z}) \big)\big(s,\Mc^{(\bar{z})}(s-t,x)\big) \big(e^{-(s-t)B^*}\nabla_x\big)_i,  \label{def_Gn}
\end{align}
with $\Mc^{(\bar{z})}(t,x)$ as in \eqref{eq:M} and, by convention, $\TT_{-1} f\equiv 0$.
\end{theorem}
%\noindent Theorem \ref{th:un_general_repres} will be proved in Appendix \ref{appendix:proof}.
Next, we sketch the proof of Theorem \ref{th:un_general_repres} that is based on the symmetry
properties of the Gaussian density $\Gamma_{0}$ in \eqref{Gamma0}, combined with an extensive use
of other very general relations such as the Duhamel's principle and the
Chapman-Kolmogorov equation. % which we recall for completeness.
%\begin{lemma}[Chapman-Kolmogorov identity]\label{lem:semigroup_property}
%Under Assumption \ref{assC}, for any $t<s<T$, $x,y\in\mathbb{R}^d$, we have
%\begin{align}\label{eq:semigroup_property}
%\int_{\mathbb{R}^d} \Gamma(t,x;s,\xi)\Gamma(s,\xi;T,y)\, d \xi = \Gamma(t,x;T,y).
%\end{align}
%\end{lemma}
%We start by recalling the operator
%\begin{align}
% %{\Mc(t,s)} &:=
% \Mc(t,s,x) = e^{(s-t)B}(x+ \Mv(s-t)\nabla_x)
%\end{align}
%as it is defined in \eqref{eq:M} (as
Since the choice of $\bar{z}$ is unimportant through this section, we drop the explicit dependence on $\bar{z}$
%omit such an index
in the following formulas.
First, we generalize formula \eqref{prop:polyn1_bis} to polynomial functions $p$ %that may depend
with time-dependent coefficients,
%also on time,
that is $p=p(t,\cdot)$ is a polynomial for every fixed $t\in\R$: this will be used to deal with the
operators $\Kc_n$ in \eqref{e1_ter} that have coefficients of this form.
\begin{proposition}\label{proposition_welldef}
For any $t,s,s_{1}\in[0,T]$, with $t<s$, $x,y \in\mathbb{R}^d$, we have
\begin{equation}\label{prop:polin_xi}
 p(s_{1},y) \Gamma_{0}(t,x;s,y)=\, p\left(s_{1},\Mc(s-t,x)\right) \Gamma_{0}(t,x;s,y) .
% \label{prop:polin_x}
% a_{\alpha,n}(s_{1},x) \Gamma_{0}(t,x;s,y)= &\,
% a_{\alpha,n}\left(s_{1},\bar{\Mc}(t,s,y)\right) \Gamma_{0}(t,x;s,y) .
\end{equation}
\end{proposition}
\begin{proof} Let us recall that operator $\Mc(t,x)$ acts only on the variable $x$. First, we prove that the components $\Mc_j (t,x)$, $i=1,\dots,d$,
commute when applied to $\Gamma_{0}=\Gamma_{0}(t,x;s,y)$ and to its derivatives (notice however
that this is not true in general when they are applied to a generic function). Notice also that
formula \eqref{prop:deriv_bis} expresses an $x$-derivative as a linear combination of
$y$-derivatives with coefficients that depend only on $t$ and $s$. This is obviously true also for
higher orders and we express it through the differential operator $S_{y}^{\b}(s-t)$, acting on
$y$, defined by
 \begin{equation}
 D^{\b}_{x}\Gamma_{0}(t,x;s,y)=S_{y}^{\b}(s-t)\Gamma_{0}(t,x;s,y).
 \end{equation}
Now we have
\begin{align}
  \Mc_i (s-t,x)\Mc_j (s-t,x)D^{\b}_{x}\Gamma_{0}&=\Mc_i (s-t,x)\Mc_j (s-t,x)S_{y}^{\b}(s-t)\Gamma_{0}
  &  &\text{(by the definition above)}\\
  &=S_{y}^{\b}(s-t)(\Mc_i (s-t,x)\Mc_j (s-t,x)\Gamma_{0}) & &\text{($S_{y}^{\b}$ and $\Mc_j$ commute)}\\
  &=S_{y}^{\b}(s-t)(\Mc_i (s-t,x)y_j\Gamma_{0}) &  &\text{(by \eqref{prop:polyn1_bis})}\\
  &=S_{y}^{\b}(s-t)(y_j\Mc_i(s-t,x)\Gamma_{0})\\
  &=S_{y}^{\b}(s-t)(y_j y_i\Gamma_{0}) & &\text{(again, by \eqref{prop:polyn1_bis})} \\
  &=\Mc_j (s-t,x)\Mc_i(s-t,x)D^{\b}_{x}\Gamma_{0}. & &\text{{(by reversing the steps above)}}
\end{align}
Since $p(s_1,\cdot)$ is a polynomial by definition, we therefore have that the operators
$p\left(s_1,\Mc(s-t,x)\right)$ are defined unambiguously when applied to $\Gamma_{0}(t,x;s,y)$ and
to its derivatives. Moreover, clearly \eqref{prop:polin_xi} is now a straightforward consequence
of \eqref{prop:polyn1_bis}.
\end{proof}
%We now recall the operator
%\begin{align}
%\Gc_{n}(t,s,x)= &\frac{1}{2}\sum_{i,j=1}^{p_0}
%\left(\TT_na_{ij}-\TT_{n-1}a_{ij}\right)\big(s,\Mc(t,s,x)\big)W^{i}(t,s)W^{j}(t,s) \\ &+
%\sum_{i=1}^{p_0}  \left(\TT_{n-1}a_{i}-\TT_{n-2}a_{i}\right)\big(s,\Mc(t,s,x)\big)W^{i}(t,s),
%\label{eq:G.def}
%\end{align}
%as it is defined in \eqref{def_Gn}. Here
%\begin{equation}
%\label{def:W} W^{i}(t,s)=\big(e^{-(s-t)B^*}\nabla_x\big)_i\qquad i=1,\dots, p_0,
%\end{equation}
\begin{remark} By Proposition \ref{proposition_welldef}, the operators $\Gc_{n}(t,s,x)$ are defined unambiguously when applied to
$\Gamma_{0}=\Gamma_{0}(t,x;s,y)$, to its derivatives and, more generally, {by the representation
formula \eqref{e10}}, to solutions of the Cauchy problem \eqref{eq:v.0.pide}.
\end{remark}
%
%By Duhamel's principle, to solve the non homogeneous Cauchy problems \eqref{eq:v.0.pide}, we need to study
The next proposition, essentially based on the symmetries of Lemma \ref{lem:symmetries}, is the
key of the proof of Theorem \ref{th:un_general_repres}. %{\blue Hereafter, throughout the rest of this section, we denote by $\Sc(\Rd)$ the usual Schwartz space of rapidly decreasing functions on $\Rd$.}
%{\blue Since the proofs are formally identical to the ones given in ...., Section 5, we omit them for sake of brevity.}
\begin{proposition}
For any $x,y\in\mathbb{R}^d$, $t<s$ and $n\in\N$ with $n\le N$, we have
\begin{equation}\label{eq:fund_prop}
\int_{\mathbb{R}^{d}} \Gamma_{0}(t,x;s,\xi){\big(\left(\Kc_n -\Kc_{n-1}\right)f\big)}(s,\xi)  d
\xi =\Gc_{n}(t,s,x)\int_{\mathbb{R}^{d}} \Gamma_{0}(t,x;s,\xi) f(\xi)  d \xi ,
%\label{eq:fund_prop1} \int_{\mathbb{R}^{d}} f(\xi) \Ac_n(s,\xi) \Gamma_{0}(s,\xi;T,y) d \xi
%&=\bar{\Gc}^{y}_{n}(s,T)\int_{\mathbb{R}^{d}} f(\xi) \Gamma_{0}(s,\xi;T,y)  d \xi ,
\end{equation}
for any $f\in C^2_{0}(\mathbb{R}^d)$.
%{\blue\Sc} \left( \mathbb{R}^d \right)$. {\blue Here, $\Sc(\Rd)$ denotes the usual Schwartz space of rapidly decreasing functions on $\Rd$.}%{\bf \blu{Si usava la notazione $(\Kc_n
%-\Kc_{n-1})(s,\xi)$ anche nella dimostrazione, ma secondo me...}}
%$\Gc_{n}(t,s,x)$ is the operator defined in \eqref{def_Gn}. %Furthermore, the following relation holds:
%\begin{align}\label{eq:fund_prop2}
%\Gc^{x}_{n}(t,s)\Gamma_{0}(t,x;T,y)=\bar{\Gc}^{y}_{n}(s,T)\Gamma_{0}(t,x;T,y).
%\end{align}
\end{proposition}
\begin{proof}
To keep formulas at a reasonable size we suppose that the functions $a_i$, $i=1,\dots,p_0,$ in
\eqref{Ac} are identically zero. By the definition \eqref{e1_ter} %such operators
we have
\begin{align}
 &\int_{\mathbb{R}^d} \Gamma_{0}(t,x;s,\xi)\big((\Kc_n -\Kc_{n-1})f\big)(s,\xi)  d \xi \\
 &= \frac{1}{2}\sum_{i,j=1}^{p_0} \int_{\mathbb{R}^d}
 \left(\TT_n\left(a_{ij},\zb\right)-\TT_{n-1}(a_{ij},\zb)\right)(s,\xi) \Gamma_{0}(t,x;s,\xi)\p_{\xi_i\xi_j} f(\xi)d \xi\\
 &= \frac{1}{2}\sum_{i,j=1}^{p_0}
 \left(\TT_n\left(a_{ij},\zb\right)-\TT_{n-1}(a_{ij},\zb)\right)\big(s,\Mc(s-t,x)\big) \int_{\mathbb{R}^d}
 \Gamma_{0}(t,x,y;s,\xi,\omega) \p_{\xi_i\xi_j} f(\xi)  d \xi &&\text{(by
 \eqref{prop:polin_xi})}\\
 &= \frac{1}{2}\sum_{i,j=1}^{p_0}\left(\TT_n\left(a_{ij},\zb\right)-\TT_{n-1}(a_{ij},\zb)\right)\big(s,\Mc(s-t,x)\big)
 \int_{\mathbb{R}^d} \p_{\xi_i\xi_j} \Gamma_{0}(t,x;s,\xi) f(\xi)  d \xi  &&\text{(by parts)}\\
% &= \frac{1}{2}\sum_{i,j=1}^{p_0}\left(\TT_n\left(a_{ij},\zb\right)-\TT_{n-1}(a_{ij},\zb)\right)
% \big(s,\Mc(s-t,x)\big){\blue \mathcal{W}_{i}(s-t)\mathcal{W}_{j}(s-t) ??}
% \int_{\mathbb{R}^d}  \Gamma_{0}(t,x;s,\xi) f(\xi)  d \xi %&&\text{{(by definition of $\Gc_{n}$)}}
% &&\text{(by \eqref{prop:deriv_bis})}
% \\
 &= \Gc_{n}(t,s,x) \int_{\mathbb{R}^d}  \Gamma_{0}(t,x;s,\xi) f(\xi) d \xi. &&\text{(by \eqref{prop:deriv_bis} and
 \eqref{def_Gn})}
\end{align}
\end{proof}
The proof of Theorem
\ref{th:un_general_repres} consists of mostly formal and tedious computations that are totally
analogous to those given for the parabolic case in Section 5 in \cite{LPP4}.
% Indeed, once Lemma 5.2 there is replaced with our Lemma \ref{lem:symmetries}, one can repeat the same proof.
This may not be surprising since our framework contains the parabolic one as a special case. Therefore, we only give a proof for
%prove  Theorem \ref{th:un_general_repres} for
$n=1$, %This allows us to avoid almost all the computations and technicalities but
which still sheds light on the origin of the operators $\Lc_n$.
%The interested reader can found the complete proof in Section 5 in \cite{LPP4}.

By definition, $u_1$ is the solution of the Cauchy problem \eqref{eq:v.n.pide} with $n=1$. By
Duhamel's principle we have
\begin{align}
&u_1(t,x)=\int_t^T \int_{\mathbb{R}^d} \Gamma_{0}(t,x;s,\xi)\left((\Kc_1-\Kc_0) u_0\right)(s,\xi)
\dd\xi d s
\\ &=\int_t^T \Gc_{1}(t,s,x) \int_{\mathbb{R}^d} \Gamma_{0}(t,x;s,\xi)\, u_0(s,\xi) d \xi {d s }& &\text{(by \eqref{eq:fund_prop} with $n=1$)} \\ &=\int_t^T \Gc_{1}(t,s,x)
\int_{\mathbb{R}^d} \Gamma_{0}(t,x;s,\xi) \int_{\mathbb{R}^d} \Gamma_{0}(s,\xi;T,y) \phi(y) d y {d
\xi  d s} & &\text{(by \eqref{e10})} \\ &=  \int_t^T \Gc_{1}(t,s,x) \int_{\mathbb{R}^d} \phi(y)
\int_{\mathbb{R}^d} \Gamma_{0}(t,x;s,\xi)\, \Gamma_{0}(s,\xi;T,y) d \xi { d y d s } &
&\text{(Fubini's theorem)} \\ &=\int_t^T \Gc_{1}(t,s,x) d s\, u_0(t,x) &
&\text{(Chapman-Kolmogorov and \eqref{e10})}\\ &=\Lc_1(t,T,x) u_0(t,x). & &\text{(by
\eqref{eq:def_Ln}) }
\end{align}
$\hfill \Box$
\section{Proof of Proposition \ref{prop:der_un}}\label{prel}
In this section we prove some preliminary estimates
on the spatial derivatives of solutions of constant coefficient-Kolmogorov operators: in
particular, we prove estimates for the derivatives of $u_{n}^{(\zb)}$ defined by
\eqref{eq:v.0.pide}-\eqref{eq:v.n.pide}. Throughout this section $\zb\in\Rdd$ is fixed.
%The quality of such estimates depends on
%the regularity of the terminal data $\psi$.
\begin{proposition}\label{l1}
Let $k\in[0,2r+1]$, $\b\in\N_{0}^{d}$ with $|\b|_B>0$. If $\psi\in C_{B}^{k}(\R^{d})$ then the
solution $u^{(\zb)}_0$ of the Cauchy problem \eqref{eq:v.0.pide} satisfies
\begin{align}\label{e1bb}
  \big|D_{x}^{\beta}u^{(\zb)}_0(t,x)\big|\le C (T-t)^{\frac{k-|\b|_B}{2}},\qquad 0\leq t<T,\ x\in\mathbb{R}^d,
\end{align}
where $C$ is a positive constant that depends only on $M,\m,B,T,\b$ and
$\|\psi\|_{C^{k}_{B}(\R^{d})}$.
%$\left\|\psi\right\|_{C^{m,k,1}_{B}}$.
\end{proposition}
\begin{proof}
We prove the case $k\in \,]0,2r+1]$ since the case $k=0$ is straightforward. We first note that,
since $\Gamma_{0}^{(\zb)}$ is a density and $|\b|_{B}>0$, we have
\begin{align}\label{eq:integral_polynom_derivat}
 D_x^{\ggg}\int\limits_{\mathbb{R}^d}\Gamma^{(\zb)}_{0}(t,x;T,y)d y=0.
\end{align}
and therefore
\begin{align}
 D_x^{\beta}  u_0^{(\zb)}(t,x)&=  \int\limits_{\mathbb{R}^d} \psi(y) D_x^{\beta} \Gamma_{0}^{(\zb)}(t,x;T,y) d y\\
  &=\int\limits_{\mathbb{R}^d} \left(\psi(y)- \psi\big(e^{(T-t)B}x\big)\right) D_x^{\beta}\Gamma_{0}^{(\zb)}(t,x;T,y) d
  y.
\end{align}
Since $\psi\in C_{B}^{k}(\R^{d})$, we obtain
\begin{align}
 \left|D_x^{\beta}  u_0^{(\zb)}(t,x)\right| &\leq \|\psi\|_{C^{k}_{B}(\R^{d})} \int\limits_{\mathbb{R}^d} \big[y-e^{(T-t)B}x\big]^{k}_B \left|D_x^{\beta} \Gamma_{0}^{(\zb)}(t,x;T,y)\right| d
 y\\
 &\leq C\|\psi\|_{C^{k}_{B}(\R^{d})}(T-t)^{\frac{k-|\b|_{B}}{2}} \int\limits_{\mathbb{R}^d} \Gamma^{{2 M}}(t,x;T,y) d
 y,
\end{align}
where the second inequality follows from a direct estimate on the derivatives of
$\Gamma_{0}^{(\zb)}$ (see, for example, Section 2 in \cite{Polidoro1994}) and $\Gamma^{2 M}$ is
the fundamental solution of the Kolmogorov operator {$\Kc^{2M}$ as defined in}
\eqref{eq:heatoperator}.
%\eqref{Gaua} and from
%\begin{equation}
%\label{eq:integral_gamma} \int\limits_{\mathbb{R}^d} \Gamma^{M+\varepsilon}(t,x;T,y)\, d y=1.
%\end{equation}
\end{proof}
In the next lemmas we will use the following result proved in \cite{LanconelliPolidoro1994}.
\begin{lemma}
\label{lem:matrix_homogeneity} %$\zb \in [0,T]\times D$ and  $t>0$ the matrices
%$\Cv_{\zb}(t),\Mv_{\zb}(t)$, and $e^{tB}$ as defined in \eqref{eq:covariance_mean} and
%\eqref{eq:M}, satisfy
The following homogeneity relations hold
\begin{align}\label{eq:matrix_homogeneity}
 \Cv_{\zb}(t)=& D_0(\sqrt{t})\Cv_{\zb}(1)D_0(\sqrt{t}), \\
 \label{eq:matrix_homogeneity_c}
 \Mv_{\zb}(t)=& D_0(\sqrt{t})\Mv_{\zb}(1)D_0(\sqrt{t}), \\ \label{eq:matrix_homogeneity_b}
 e^{tB}=& D_0(\sqrt{t})e^{B}D_0\left(\frac{1}{\sqrt{t}}\right),
\end{align}
for any $t>0$.
\end{lemma}
%The proof of Theorem \ref{th:error_estimates_taylor} will be achieved through time estimates on
%the coefficients of the operators in \eqref{eq:def_Ln}. We begin with the operators $W^{i}(t,s)$
%defined in \eqref{def:W}. Moreover, for a multindex $\a=e_i+e_j$ with $e_i,e_j$ elements of the
%canonical basis of $\Rd$ we set $W^{\a}(t,s)=W^{i}(t,s)W^{j}(t,s)$.
\begin{notation}\label{not}
From now to the end of this section, we use the Greek letters $\a,\b,\g,\d,\nu$ to denote
multi-indexes in $\N_{0}^{d}$, and $|\a|=\sum_{i=1}^{d}\a_{i}$ is the standard Euclidean height of
$\a$. To simplify notations, if $I$ is any family of indexes, we use the unconventional notation
\begin{equation}\label{sumc}
  \sum_{\ell\in I}^{\bullet}\pi_{\ell} =\sum_{\ell\in I}c_{\ell}\pi_{\ell}
\end{equation}
for a sum where the constants $c_{\ell}$ depend only on $\zb, B,N,T,a_{ij},a_{i}$ and are
uniformly bounded by a constant that depends only on $M,\m,B, T,N$ and $B$.
\end{notation}
\begin{lemma}
%\label{rem:operators_M_W}
\label{lem:W_representation} Let
\begin{equation} \label{def:W}
 \mathcal{W}(t)=e^{-tB^*}\nabla_x, \qquad t\in\R, %i=1,\dots, d,
\end{equation}
denote the differential operators appearing in \eqref{def_Gn} and by $\mathcal{W}^{\a}(t)$ the
composition\footnote{Operator $\mathcal{W}^{\a}(t)$ in \eqref{noy} is well defined since the
components of $\mathcal{W}(t)$ commute.}
\begin{equation}\label{noy}
  \mathcal{W}^{\a}(t)=\mathcal{W}^{\a_{1}}_{1}(t)\cdots \mathcal{W}^{\a_{d}}_{d}(t).
\end{equation}
The following representation holds true:
\begin{equation}
 \mathcal{W}^{\b}(t)= \sum_{\substack{|\a|=|\b|\\
 |\a|_{B}\ge |\b|_{B}}}^{\bullet} t^{\frac{|\a|_B-|\b|_B}{2}}D_x^{\a}.
\end{equation}
%where %each function $w_{\a,\b}$ is either zero or
%\begin{equation}
%  w_{\a,\b}(t)=c_{\a,\b}\, t^{\frac{|\b|_B-|\a|_B}{2}},
%\end{equation}
%for suitable constants $c_{\a,\b}$ that depend on $B$. %, with $c_{\a,\b}=0$ if $|\a|_{B}>|\b|_{B}$.
\end{lemma}
%Thanks to Proposition \ref{l1}, derivatives of the solution to the constant coefficient problem \eqref{eqpde} can be estimated by powers of time. We would like to understand what it happens when we apply operators such $W^{i}(t,s)$ and $\Mc^{x}(t,s)$ that also depends  on time. It turns out that, for an homogeneous matrix $B$, those operators depend polynomially on both $(s-t)$ and the derivatives. Even more importantly, there is a  "pairing" between higher powers of time and higher intrinsic order derivatives that, to the end the estimate, allow us to treat the operators as a single derivative.
\begin{proof}
It suffices to prove the statement for a single $\mathcal{W}_i(t)$. Using the relations in Lemma
\ref{lem:matrix_homogeneity}, we have
\begin{align}
 \mathcal{W}_{i}(t) %=& (e^{-tB^*}\nabla_x)_i \\
          &= \sum_{j=1}^d D_0\left(\frac{1}{\sqrt{t}}\right)_{i i}e^{-B^*}_{ij}D_0\big(\sqrt{t}\big)_{jj}\p_{x_j}\\
          &= t^{-\frac{\s_i}{2}}\sum_{j=1}^d e^{-B^*}_{ij}t^{\frac{\s_j}{2}}\p_{x_j},
\end{align}
with $\s_{i}$ as in \eqref{eq:dilation_exponents}. The result follows noting that the intrinsic
order of $\p_{x_j}$ is exactly $\s_j$. Moreover, as the matrix $e^{-B^*}$ is upper triangular the
sum actually ranges over $j=i,\dots, d$ and thus $\s_j-\s_i$ is always a nonnegative integer.
\end{proof}
Next step is the study of the operator $\Mc^{(\bar{z})}(t,x)$: we recall that, by Proposition
\ref{proposition_welldef}, the components of $\Mc^{(\bar{z})}(t,x)$ commute when applied to
$\G_{0}^{(\zb)}$ and more generally to $u^{(\zb)}_{n}$ and its derivatives.
%As explained in Proposition
%\ref{proposition_welldef}, it will be plugged in a (Taylor) polynomial and thus we state the
%result for  the more general operator  $\Mc^{(\bar{z})}(t,s,x)-e^{(s-\bar{t})B}\xb$.
\begin{lemma}
\label{lem:Mx_representation}
%%Fix a point $\zb=(\bar{t,\xb})$ and consider the Taylor expansion centered at $\zb$
%Let $\zb=(\bar{t},\xb)$ be the base point for the Taylor expansion.
%For the operator\\ $\left(\Mc^{(\bar{z})}(t,s,x)-e^{(s-\bar{t})B}\xb\right)^{\b}$
%, appearing in the expression of $\Gc^{(\zb)}_{n}(t,s,x)$ in \eqref{def_Gn}
%The following representation holds true:
For any $\b\in\N_{0}^{d}$, we have
 \begin{equation}\label{eq:ste100}
 \left(\Mc^{(\zb)}(s-t,x)-e^{\left(s-\bar{t}\right)B}\xb\right)^{\b} = %\left(e^{(s-t)B}(x -e^{(t-\bar{t})B}\xb +M(s-t)\nabla_x\right)^{\b}=
 \sum_{\substack{%|\a|\leq |\b|\\
 |\d|+|\a|\leq |\b|\\ |\d|_B-|\a|_B\le |\b|_B}}^{\bullet}
 %c_{\a,\d,\b}\,
 (s-t)^{\frac{|\b|_B+|\a|_B-|\d|_B}{2}}\left(x-e^{\left(t-\bar{t}\right)B}\xb\right)^{\d}D_x^{\a}.
\end{equation}
%where %Moreover, each function $f^{\b}_{\a,\d}$ is either zero or
\end{lemma}
\begin{proof}
First of all, let us note that
\begin{equation}
 \Mc^{(\zb)}(s-t,x)-e^{(s-\bar{t})B}\xb=e^{(s-t)B}\left(x -e^{(t-\bar{t})B}\xb +\Mv_{\zb}(s-t)\nabla_x\right),
\end{equation}
and %, as derivatives act on $x$, we can suppose
it is not restrictive to take $\xb=0$ and $t=0$.
We proceed now by induction on $|\beta|$. If $|\beta|=1$ then $\b=\mathbf{e}_i$ where
$\mathbf{e}_i$ is the $i$-th element of the canonical basis of $\Rd$. A direct computation shows
\begin{align}
 \big(\Mc^{(\zb)}(s,x)\big)^{\mathbf{e}_{i}} & = \sum_{|\d|=1 \atop |\d|_B \leq |\mathbf{e}_i|_B}^{\bullet}
 s^{\frac{|\mathbf{e}_i|_B - |\d|_B}{2}}  \big(x^{\d} +(\Mv_{\zb}(s)\nabla_x)^{\d}\big)&&\text{(by \eqref{eq:matrix_homogeneity_b})}\\
 & = \sum_{|\d|=1 \atop |\d|_B \leq |\mathbf{e}_i|_B}^{\bullet} s^{\frac{|\mathbf{e}_i|_B - |\d|_B}{2}}  \bigg(x^{\d} + s^{\frac{|\d|_B}{2}}
 \sum_{|\nu| = 1}^{\bullet} s^{\frac{|\nu|_B}{2}} D^{\nu}_x \bigg),&&\text{(by \eqref{eq:matrix_homogeneity_c})}
\label{eq:ste101}
\end{align}
which proves \eqref{eq:ste100} with $\b=\mathbf{e}_{i}$. We now assume the {statement} to hold for
$|\beta|\le n$, and prove it true for $\beta + \mathbf{e}_i$. By inductive hypothesis applied to
both $\b$ and $\mathbf{e}_i$ we get
\begin{align}
 \big(\Mc^{(\zb)}(s,x)\big)^{\b+\mathbf{e}_i} & = \sum_{\substack{|\d^1|+|\a^1|\leq 1\\ |\d^1|_B-|\a^1|_B\le |{\bf e}_i|_B}}^{\bullet}
  \sum_{\substack{|\d^2|+|\a^2|\leq |\b|\\ |\d^2|_B-|\a^2|_B\le |\b|_B}}^{\bullet} s^{\frac{|\mathbf{e}_i|_B +|\a^1|_B - | \d^1|_B}{2}}s^{\frac{|\b|_B +|\a^2|_B - | \d^2|_B}{2}}  x^{ \d^1}
   D_x^{\a^1}\left(x^{ \d^2} D_x^{\a^1}\right)\\
 & = \sum_{\substack{ |\d|+|\a|\leq |\b + \mathbf{e}_i| \\ |\d|_B-|\a|_B\le |\b + \mathbf{e}_i|_B}}^{\bullet}  s^{\frac{|\b + \mathbf{e}_i|_B +|\a|_B - | \d|_B}{2}}  x^{ \d}
 D_x^{\a},\qquad\ \text{(setting $\d=\d^1+\d^2$ and $\a=\a^1+\a^2$).}
\end{align}
\end{proof}
\begin{lemma} For any $n\in\N$, with $n\le N$, we have the following representation
%studying $\Gc^{\zb}_{n}(t,s)$...
\begin{equation}\label{eq:estimate_Gc}
 \Gc^{(\zb)}_{n}(t,s,x)=%\sum_{|\b|_B=n-1}^n
 \sum_{(\a,\d)\in I_{n}}^{\bullet} (s-t)^{\frac{|\a|_B-|\d|_B +
 n-2}{2}}(x-e^{(t-\bar{t})B}\xb)^{\d}D_x^{\a},
\end{equation}
%for some functions $g_{\a,\d}^{k,\b}$ that either are identically zero or
%\begin{equation}
%\label{eq:time_beh_g}
%g_{\a,\d}^{\b}(s-t)= d_{\a,\d}^{\b}(s-t)^{\frac{1}{2}(|\a|_B-|\d|_B + n-2)},
%\end{equation}
where
  $$I_{n}=\{(\a,\d)\in \N_{0}^{d}\times \N_{0}^{d}\mid 1\le |\a|\le n+2,\, |\d|_{B}\le n,\, |\a|_B-|\d|_B + n-2\ge 0\}.$$
\end{lemma}
\begin{proof}
Using the definition of $\Gc^{(\zb)}_{n}(t,s,x)$ in \eqref{def_Gn}, the proof is a straightforward
application of Lemmas \ref{lem:W_representation} and \ref{lem:Mx_representation}.
\end{proof}
\begin{lemma}\label{lem:Ln}
For any $n\in\N$, with $n\le N$, we have the following representation
\begin{equation}\label{eq:estimate_Ln}
 \Lc^{(\zb)}_{n}(t,T,x)=%\sum_{|\b|_B=n-1}^n
 \sum_{(\a,\d)\in J_{n}}^{\bullet} (T-t)^{\frac{|\a|_B-|\d|_B + n}{2}}(x-e^{(t-\bar{t})B}\xb)^{\d}D_x^{\a},
\end{equation}
where
\begin{equation}\label{Jn}
  J_{n}=\{(\a,\d)\in \N_{0}^{d}\times \N_{0}^{d}\mid 1\le |\a|\le 3n,\, |\d|_{B}\le n,\, |\a|_B-|\d|_B + n\ge 0\}.
\end{equation}
\end{lemma}
\begin{proof}
For greater convenience we recall the expression of $\Lc^{(\bar{z})}_n(t,T,x)$ as given in
\eqref{eq:def_Ln}:
\begin{align}
 \Lc^{(\bar{z})}_n(t,T,x)=\sum_{h=1}^{n}\sum_{i\in I_{n,h}}L_{h,i}(t,T,x), %\int_{t}^T d s_1 \int_{s_1}^T d s_2 \cdots \int_{s_{h-1}}^T d s_h \Gc^{(\bar{z})}_{i_{1}}(t,s_1,x) \cdots \Gc^{(\bar{z})}_{i_{h}}(t,s_h,x),
\end{align}
where
\begin{align}%\label{eq:def_Ln_bis}
  L_{h,i}(t,T,x):=&\int_{t}^T d s_1 \int_{s_1}^T d s_2 \cdots \int_{s_{h-1}}^T d s_h \Gc^{(\bar{z})}_{i_{1}}(t,s_1,x) \cdots
  \Gc^{(\bar{z})}_{i_{h}}(t,s_h,x),
\end{align}
and $I_{n,h}= \{i=(i_{1},\dots,i_{h})\in\mathbb{N}^{h} \mid i_{1}+\dots+i_{h}=n\}$, for $1\le h
\le n$. We prove that, for fixed $h\in \{1,\dots,n\}$ and $i\in I_{n,h}$ it holds
\begin{align}\label{c}
 L_{h,i}(t,T,x)=\sum_{(\a,\d)\in J_{n}}^{\bullet} (T-t)^{\frac{|\a|_B-|\d|_B + n}{2}}(x-e^{(t-\bar{t})B}\xb)^{\d}D_x^{\a},
\end{align}
the result will then readily follow.
We only consider the case $\xb=0$. Plugging equation \eqref{eq:estimate_Gc} into the definition of
$L_{h,i}$ we obtain
\begin{multline}
   L_{h,i}(t,T,x) = \sum_{(\a^1,\d^1)\in I_{i_1}}^{\bullet}\cdots\sum_{(\a^{h},\d^{h})\in I_{i_h}}^{\bullet} x^{\d^{1}}D_x^{\a^{1}}
   \left(x^{\d^2}D_x^{\a^2}\left(\cdots\left(x^{\d^h}D_x^{\a^h}\right)\right)\right)\times\\
       \times\int_{t}^T \cdots \int_{s_{h-1}}^T \prod_{j=1}^h(s_j-t)^{\frac{|\a^j|_B-|\d^j|_B+i_j-2}{2}} d s_1\cdots d s_h.
\end{multline}
Now, setting $\a=\a^1+\cdots +\a^h$, $\d=\d^1+\cdots +\d^h$ and recalling that $i_1+\cdots
+i_h=n$, the integral above can be easily computed to be equal to
\begin{equation}
(T-t)^{\frac{|\a|_B-|\d|_B+n}{2}},
\end{equation}
times a constant. The {statement} follows applying Leibniz rule %to move derivatives to the right
and noticing that $(\a,\d)\in J_n$ if $(\a^j,\d^j)\in I_{i_j}$ for $j=1,\dots,h$.
\end{proof}
\begin{proof}[Proof of Proposition \ref{prop:der_un}]
By \eqref{eq:un}-\eqref{eq:estimate_Ln}, %, and applying Leibniz rule,
we get
\begin{align}
 D_x^{\b}u^{(\zb)}_n(t,x) & = D_x^{\b} \sum_{(\a,\d)\in J_{n}}^{\bullet} (T-t)^{\frac{|\a|_B-|\d|_B
 + n}{2}}\big(x-e^{(t-\bar{t})B}\xb\big)^{\d}D_x^{\a} u^{(\zb)}_0(t,x)
\intertext{(by applying Leibniz rule and reordering the indexes of $J_{n}$ in \eqref{Jn})}
 & =\sum_{\substack{(\a,\d)\in J_{n}\\
 \nu\le \min\{\b,\d\}}}^{\bullet} (T-t)^{\frac{|\a|_B-|\d|_B + n}{2}}\big(x-e^{(t-\bar{t})B}\xb\big)^{\d-\nu}D_x^{\a+\b-\nu}
 u^{(\zb)}_0(t,x),
\end{align}
where $\nu\le \min\{\b,\d\}$ means that $\nu_{i}\le \min\{\b_{i},\d_{i}\}$ for any $i=1,\dots,d$.
Now, by applying Proposition \ref{l1} and the property
\begin{equation}
% \big|\big(x-e^{(t-\bar{t})B}\xb\big)^{\delta}\big|\leq \big[x-e^{(t-\bar{t})B}\xb\big]_B^{|\delta|_B},
 \big|y^{\d}\big|=\prod_{i=1}^{d}\left|y_{i}\right|^{\d_{i}}\le
 \prod_{i=1}^{d}[y]_{B}^{\s_{i}\d_{i}}=[y]_{B}^{|\d|_{B}},\qquad y\in\R^{d},
\end{equation}
we obtain
\begin{align}
 \big|D_x^{\b}u^{(\zb)}_n(t,x)\big| &\leq \sum_{\substack{(\a,\d)\in J_{n}\\
 \nu\le \min\{\b,\d\}}}^{\bullet} (T-t)^{\frac{-|\d|_B + n
 +k  - |\b|_B+|\nu|_B}{2}} \big[x-e^{(t-\bar{t})B}\xb\big]_B^{|\delta|_B-|\nu|_B}\\
 & = \sum_{0\le m\le n}^{\bullet}
(T-t)^{\frac{-m + n +k - |\b|_B}{2}} \big[x-e^{(t-\bar{t})B}\xb\big]_B^{m},
\end{align}
and the statement follows by the elementary inequality {
\begin{equation}
 a^m b^{n-m}\leq a^n+b^n,\qquad  a,b\in \R_{>0},\ 0\le m\le n.
\end{equation}
}
\end{proof}
%%%%%%%%%%%%%%%%%%%%%%%%%%%%%%%%%%%%%%%%%%%%%%%%%%%%
%
%           Bibliography
%
%%%%%%%%%%%%%%%%%%%%%%%%%%%%%%%%%%%%%%%%%%%%%%%%%%%%
\bibliographystyle{siam}
\bibliography{LPP-bib}
\end{document}